\documentclass{article}

\usepackage{latexsym}
\usepackage{a4wide}
\usepackage{amscd}
\usepackage{graphics}
\usepackage{amsmath}
\usepackage{amssymb}
\usepackage{mathrsfs}
\input xy
\xyoption{all}

\newcommand{\Z}{\mathbb{Z}}                     % the integer numbers
\newcommand{\R}{\mathbb{R}}                     % the real line
\newcommand{\set}[2]{\left\{{#1}\mid{#2}\right\}}       % the set
\newcommand{\qed}{\hfill $\Box$ \bigskip}       % end of proof
\newcommand{\sgn}{\mathrm{sgn\,}}               % signum
\newcommand{\codim}{\mathrm{codim}}           % codimension
\newcommand{\graf}{\mathrm{graph\,}}            % graph
\newcommand{\ran}{\mathrm{ran\,}}       % range
\newtheorem{thm}{\sc Theorem}[section]      % numbered within each section
\newtheorem{cor}[thm]{\sc Corollary}        % numbered along with Theorem
\newtheorem{lem}[thm]{\sc Lemma}            % numbered along with Theorem
\newtheorem{prop}[thm]{\sc  Proposition}     % numbered along with Theorem
\newtheorem{defn}[thm]{\sc Definition}      % numbered along with Theorem
\newtheorem{rem}[thm]{\sc Remark}       % numbered along with Theorem
\title{The homology of path spaces and Floer homology with conormal 
boundary conditions}

\author{Alberto Abbondandolo\thanks{The first author was partially supported by a {\em Humboldt Research Fellowship for Experienced Researchers}.}, Alessandro Portaluri\thanks{The second author was partially supported by the MIUR project {\em Variational Methods and Nonlinear Differential Equations}.}, Matthias Schwarz\thanks{The third author was partially supported by the DFG grant SCHW 892/2-3.}}

\date{December 23, 2008}

\begin{document}

\maketitle

\centerline{\em Dedicated to Felix E.\ Browder}

\begin{abstract}
We define the Floer complex for Hamiltonian orbits  on the cotangent bundle of a compact manifold which satisfy non-local conormal boundary conditions. 
We prove that the homology of this chain complex
is isomorphic to the singular homology of the natural path space
associated to the boundary conditions. 
\end{abstract} 

\section*{Introduction}

Let $H:[0,1]\times T^*M \rightarrow \R$ be a smooth time-dependent
Hamiltonian on the cotangent bundle of a compact manifold $M$, and let
$X_H$ be the Hamiltonian vector field induced by $H$ and by the
standard symplectic structure of $T^*M$. The aim of this paper is to
define the {\em Floer complex} for 
the orbits of $X_H$ satisfying {\em non-local conormal boundary
  conditions}, and to compute its homology. More precisely, we fix a
compact submanifold $Q$ of $M^2 = M \times M$ and we look for
solutions $x:[0,1]\rightarrow T^*M$ of the equation
\[
x'(t) = X_H(t,x(t)),
\]
such that the pair $(x(0),-x(1))$ belongs to the {\em conormal bundle}
$N^* Q$ of $Q$ in $T^* M^2$. We recall that the conormal bundle of a
submanifold $Q$ of the manifold $N$ (here $N=M^2$) is the set of
covectors in $T^*N$ which are based at points of $Q$ and vanish on the
tangent space of $Q$. Conormal bundles are Lagrangian submanifolds of
the cotangent bundle, and we show that they can be characterized as 
those mid-dimensional submanifolds of $T^*N$ on which the Liouville form vanishes identically (see Proposition \ref{conorcar} for the precise statement). 

When $Q=Q_0 \times Q_1$ is the product of two submanifolds $Q_0$,
$Q_1$ of $M$, the above boundary condition is a local one, requiring
that $x(0)\in N^* Q_0$ and $x(1) \in N^* Q_1$. Extreme cases are given
by $Q_0$ and/or $Q_1$ equal to a point  or equal to $M$: since the
conormal bundle of a point $q\in M$ is the fiber $T_{q}^* M$, the
first case produces a {\em Dirichlet boundary condition}, while since
$N^* M$ is the zero section in $T^*M$, the second one corresponds to a
{\em Neumann boundary condition}. A non-local example is given by
$Q=\Delta$, the diagonal in $M\times M$, inducing the periodic orbit
problem (provided that $H$ can be extended to a smooth function on $\R
\times T^*M$ which is $1$-periodic in time). Another interesting
choice is the one producing the figure-eight problem: $M$ is itself a
product $O \times O$, and $Q$ is the subset of $M ^2 = O ^4$
consisting of points of the form $(o,o,o,o)$, $o\in O$. The Floer
complex for the figure-eight problem enters in the factorization of
the pair-of-pants product on $T^*O$ (see \cite{as08}).

The set of solutions of the above non-local boundary value Hamiltonian
problem is denoted by $\mathscr{S}^Q(H)$. If $H$ is generic, all of these
solutions are non-degenerate, meaning that the linearized problem has
no non-zero solutions, and $\mathscr{S}^Q(H)$ is at most countable
(and in general infinite). The free Abelian group generated by the
elements of $\mathscr{S}^Q(H)$ is denoted by $F^Q(H)$. This group can
be graded by the Maslov index of the path $\lambda$ of Lagrangian
subspaces of $T^* (\R^n \times \R^n)$ which is produced by the graph of the
differential of the Hamiltonian flow along $x\in \mathscr{S}^Q(H)$,
with respect to the the tangent space of $N^*Q$, after a suitable
symplectic trivialization of $x^*(TT^*M)\cong [0,1] \times T^ *\R^n$,
$n=\dim M$. Our first result is that this Maslov index does not depend
on the choice of this trivialization, provided that the trivialization
preserves the vertical subbundle and maps the tangent space of $N^* Q$
at $(x(0),-x(1))$ into the conormal space $N^* W$ of some linear
subspace  $W\subset \R^n \times \R^n$. See Section \ref{hscbcbs} for
the precise statement. 

When the Hamiltonian $H$ is the Fenchel-dual of a fiber-wise strictly convex Lagrangian 
$L:[0,1] \times TM \rightarrow \R$, the $M$-projection of the orbit $x\in \mathscr{S}^Q(H)$ is an
extremal curve $\gamma$ of   the Lagrangian action functional
\[
\mathbb{S}_L(\gamma) = \int_0^1 L(t,\gamma(t),\gamma'(t))\, dt,
\]
subject to the non-local constraint $(\gamma(0),\gamma(1))\in Q$. In
this case,  
a theorem of Duistermaat \cite{dui76} can be used to show that the
above Maslov index $\mu(\lambda,N^* W)$ coincides up to a shift with
the Morse index $i^Q(\gamma)$ of $\gamma$, where $\gamma$ is seen as a critical point of  
$\mathbb{S}_L$ in the space of paths on $M$ satisfying the above
non-local constraint.
Indeed, in Section \ref{morindsec} we prove the identity
\[
i^Q(\gamma) = \mu(\lambda,N^* W)  + \frac{1}{2} ( \dim Q - \dim M) - 
\frac{1}{2} \nu^Q(x),
\]
where $\nu^Q(x)$ denotes the nullity of $x$, i.e.\ the dimension of
the space of solutions of the linearization at $x$ of the non-local 
boundary value problem.
This formula suggests that we should incorporate the shift $(\dim Q - \dim M)/2$
into 
the grading of $F^Q(H)$, which is then graded by the index
\[
\mu^Q(x) := \mu(\lambda,N^* W)  + \frac{1}{2} ( \dim Q - \dim M).
\]
This number is indeed an integer if $x$ is non-degenerate.
When the Hamiltonian $H$ satisfies suitable growth conditions on the
fibers 
of $T^* M$, the solutions of the Floer equation
\[
\partial_s u + J(t,u) (\partial_t u - X_H(t,u)) = 0
\]
 on the strip $\R \times [0,1]$ with coordinates $(s,t)$, satisfying
 the boundary condition $(u(s,0),-u(s,1))\in N^* Q$ for every real
 number $s$ and converging to two given elements of
 $\mathscr{S}^Q(H)$  for $s\rightarrow \pm \infty$, form a pre-compact
 space. Here $J$ is a time-dependent almost complex structure on $T^*M$, compatible with the symplectic structure and $C^0$-close enough to the almost complex structure 
induced by a  Riemannian metric on $M$. Assuming also that the elements of  
$\mathscr{S}^Q(H)$ are non-degenerate, a standard counting process
 defines  a boundary operator on the graded group $F^Q_*(H)$, which
 then  carries the structure of a chain complex, called the {\em Floer complex} of 
 $(T^*M,Q,H,J)$. This free chain complex is well-defined up to chain
 isomorphism.  
 
 Changing the Hamiltonian $H$ produces chain homotopy equivalent Floer
 complexes, so  in order to compute the homology of the Floer complex
 we can  assume that $H$ is the Fenchel-dual of a  
Lagrangian $L$ which is positively quadratic in the velocities. In this case, we prove that the Floer complex of 
$(T^*M,Q,H,J)$ is isomorphic to the {\em Morse complex} of the
 Lagrangian  action functional $\mathbb{S}_L$ on the Hilbert manifold
 consisting of the absolutely continuous paths $\gamma:[0,1]
 \rightarrow M$  with square-integrable derivative and such that the
 pair 
$(\gamma(0),\gamma(1))$ is in $Q$. The latter space is homotopically  
equivalent to the path space
\[
P_Q([0,1],M) = \set{\gamma: [0,1] \rightarrow M}{\gamma \mbox{ is
    continuous  and } (\gamma(0),\gamma(1)) \in Q},
\]
so Morse theory for $\mathbb{S}_L$ implies that the homology of the
Floer complex of  of $(T^*M,Q,H,J)$ is isomorphic to the singular
homology of $P_Q([0,1],M)$. The isomorphism between the Morse and the
Floer complexes is constructed by counting the space of solutions of a
mixed problem, obtained by coupling the negative gradient flow of
$\mathbb{S}_L$ with respect to a $W^{1,2}$-metric with the Floer
equation on the half-strip $[0,+\infty[ \times [0,1]$.

These results generalize the case of Dirichlet boundary conditions ($Q$
is the singleton $\{(q_0,q_1)\}$ for some pair of points $q_0,q_1\in
M$, and $P_Q([0,1],M)$ has the homotopy type of the based loop space
of $M$) and the case of periodic boundary conditions ($Q=\Delta$, and
$P_Q([0,1],M)$ is the free loop space of $M$), studied by the first
and last author in \cite{as06}. They also generalize the results by Oh \cite{oh97}, 
concerning the case $Q=M\times S$, where $S$ is a compact submanifold of $M$ (with such a choice, the path space $P_Q([0,1],M)$ is homotopically equivalent to $S$, so one gets a finitely generated Floer homology, isomorphic to the singular homology of $S$). See \cite{vit96} and \cite{sw06} for
previous proofs of the isomorphism between the Floer homology for
periodic Hamiltonian orbits on $T^*M$ and the singular homology of the
free  loop space of $M$ (see also the review paper \cite{web05}). See also 
\cite{ng06} for the role of conormal bundles in the study of knot invariants. 

Most of the arguments from \cite{as06} readily extend to the present
more  general setting, so we just sketch them here, focusing the
analysis   on the index questions, which constitute the more original
part  of this paper.

\medskip

\paragraph{Acknowledgments.} This paper was completed while the first author was spending a one-year research period at the Max-Planck-Institut f\"ur Mathematik in den Naturwissenschaften and the Mathematisches Institut of the Universit\"at Leipzig, with a {\em Humboldt Research Fellowship for Experienced Researchers}. He wishes to thank the Max-Planck-Institut for its warm hospitality and the Alexander von Humboldt Foundation for its 
financial support.

\section{Linear preliminaries}
\label{lin}

Let $T^* \R^n = \R^n \times (\R^n)^*$ be the cotangent space of the
vector  space $\R^n$.  The {\em Liouville one-form} on $T^* \R^n$ is
the  tautological one-form $\theta_0 := p\, dq$, that is
\[
\theta_0(q,p) [(u,v)] := p[u], \quad \forall q,u \in \R^n, \; \forall
p,v  \in (\R^n)^*.
\]
Its differential
\[
\omega_0 := d\theta_0 = dp \wedge dq, \quad \omega_0
[(q_1,p_1),(q_2,p_2)]  := p_1[q_2] - p_2[q_1],
\]
is the {\em standard symplectic form} on $T^* \R^n$. The group of
linear  automorphisms of $T^* \R^n$ which preserve $\omega_0$ is the
{\em symplectic group} $\mathrm{Sp}(T^* \R^n)$. The {\em Lagrangian
  Grassmannian} $\mathscr{L}(T^* \R^n)$ is the space of all
$n$-dimensional  linear subspaces of $T^* \R^n$ on which $\omega_0$
vanishes identically.

\begin{rem}
\label{linlag}
For future reference, we recall the following description of a Lagrangian linear subspace $\lambda$ of $T^* \R^n$. Let $X$ be the linear subspace of $\R^n$ such that $\lambda \cap (\R^n \times (0)) = X \times (0)$. Choose a linear complement $Y$ of $X$ in $\R^n$, and let $(\R^n)^* = Y^{\perp} \oplus X^{\perp}$ be the corresponding decomposition of the dual of $\R^n$. Then $\lambda \cap (Y\times Y^{\perp}) = (0)$. Indeed, if $(q,p)$ is in $\lambda$ then the fact that $\lambda$ is Lagrangian and contains $X\times (0)$ implies that for every $x\in X$ we have
\[
0 = \omega [ (q,p), (x,0) ] = p[x],
\]   
so $p\in X^{\perp}$. If, in addition, $(q,p)$ is also in $Y\times Y^{\perp}$, this implies that $p=0$, and hence $q=0$ because of the definition of $X$. In particular, $\lambda$ is the graph of a linear mapping from $X\times X^{\perp}$ into $Y\times Y^{\perp}$.
\end{rem}

If $\lambda,\nu:[a,b] \rightarrow \mathscr{L}(T^* \R^n)$ are two
continuous  paths of Lagrangian subspaces, the {\em relative Maslov
  index}  $\mu(\lambda,\nu)$ is a half-integer counting the
intersections  $\lambda(t) \cap \nu(t)$ algebraically. We refer to
\cite{rs93} for the definition and the main properties of the relative
Maslov index. Here we just need to recall the formula for the relative
Maslov index $\mu(\lambda,\lambda_0)$ of a continuously differentiable
Lagrangian path $\lambda$ with respect to a constant one $\lambda_0$,
in  the case of {\em regular crossings}. Let $\lambda: [a,b]
\rightarrow  \mathscr{L}(T^* \R^n)$ be a continuously differentiable
curve,  and let $\lambda_0$ be in $\mathscr{L}(T^*\R^n)$. Fix $t\in
[a,b]$  and let $\nu_0\in \mathscr{L}(T^* \R^n)$ be a Lagrangian
complement  of $\lambda(t)$. If $s$ belongs to a suitably small
neighborhood  of $t$ in $[a,b]$, for every $\xi\in \lambda(t)$ we can
find a  unique $\eta(s)\in \nu_0$ such that $\xi+\eta(s) \in
\lambda(s)$.  The {\em crossing form} $\Gamma(\lambda,\lambda_0,t)$ at
$t$  is the quadratic form on $\lambda(t) \cap \lambda_0$ defined by
\begin{equation}
\label{crform}
\Gamma(\lambda,\lambda_0,t) : \lambda(t) \cap \lambda_0 \rightarrow
\R,  \quad \xi \mapsto \frac{d}{ds} \omega_0(\xi,\eta(s))\Bigl|_{s=t}.
\end{equation}
The number $t$ is said to be a {\em crossing} if $\lambda(t) \cap
\lambda_0  \neq (0)$, and it is called a {\em regular crossing} if the
above  quadratic form is non-degenerate. Regular crossings are
isolated, and  if $\lambda$ and $\lambda_0$ have only regular
crossings the  relative Maslov index of $\lambda$ with respect to
$\lambda_0$  is defined as
\begin{equation}
\label{relmasl}
\mu(\lambda,\lambda_0) := \frac{1}{2} \sgn \Gamma(\lambda,\lambda_0,a)
+  \sum_{a<t<b} \sgn \Gamma (\lambda,\lambda_0,t) + \frac{1}{2} \sgn
\Gamma(\lambda,\lambda_0,b),
\end{equation}
where sgn denotes the signature.

If $V$ is a linear subspace of $\R^n$, its {\em conormal space} $N^*
V$ is  the linear subspace of $T^* \R^n$ defined by
\[
N^*  V := V \times V^{\perp} = \set{(q,p) \in \R^n \times
  (\R^n)^*}{q\in V, \; p[u] = 0 \; \forall u \in V}. 
\]
Conormal spaces are Lagrangian subspaces of $T^* \R^n$. The set of all
conormal spaces is denoted by $\mathscr{N}^*(\R^n)$, 
\[
\mathscr{N}^*(\R^n) := \set{N^* V}{V\in \mathrm{Gr}(\R^n)},
\]
where $\mathrm{Gr}(\R^n)$ denotes the Grassmannian of all linear
subspaces  of $\R^n$.
The conormal space of $(0)$, $N^* (0) = (0) \times (\R^n)^*$, is
called  the {\em vertical subspace}. Note that if $\alpha$ is a linear
automorphism of $\R^n$ and $V\in \mathrm{Gr}(\R^n)$, then
\begin{equation}
\label{tn}
\left( \begin{array}{cc} \alpha^{-1} & 0 \\ 0 & \alpha^T \end{array}
\right)  N^* V = \alpha^{-1} V \times \alpha^T V^{\perp} = \alpha^{-1}
V \times (\alpha^{-1} V)^{\perp} = N^* (\alpha^{-1} V),
\end{equation}
where $\alpha^T \in \mathrm{L}((\R^n)^*,(\R^n)^*)$ denotes the
transpose  of $\alpha$.

Let $C: T^* \R^n \rightarrow T^* \R^n$  be the linear involution
\[
C(q,p) := (q,-p), \quad \forall q\in \R^n, \; \forall p\in (\R^n)^*,
\]
and note that $C$ is anti-symplectic, meaning that
\[
\omega_0 (C\xi,C\eta) = - \omega_0(\xi,\eta), \quad \forall \xi,\eta
\in  T^* \R^n.
\]
In particular, $C$ maps Lagrangian subspaces into Lagrangian
subspaces.  Changing the sign of the symplectic structure changes the
sign  of the Maslov index, so the naturality property of the Maslov
index  implies that
\begin{equation}
\label{clag}
\mu(C\lambda,C\nu) = - \mu(\lambda,\nu),
\end{equation}
for every pair of continuous paths $\lambda,\nu:[0,1] \rightarrow
\mathscr{L}(T^*\R^n)$. Since conormal subspaces of $T^* \R^n$ are
$C$-invariant, we deduce the following:

\begin{prop}
\label{vanish1}
If $V,W:[0,1] \rightarrow \mathrm{Gr}(\R^n)$  are two continuous paths
in  the Grassmannian of $\R^n$, then $\mu(N^*V,N^*W)=0$.
\end{prop}

The subgroup of the symplectic automorphisms of $T^* \R^n$ which fix
the  vertical subspace is denoted by
\[
\mathrm{Sp_v}(T^* \R^n) := \set{B \in \mathrm{Sp}(T^* \R^n)}{B N^* (0)
  =  N^* (0)}.
\]
The elements of the above subgroup can be written in matrix form as
\[
B = \left( \begin{array}{cc} \alpha^{-1} & 0 \\ \beta & \alpha^T
  \end{array} 
\right),
\]
where $\alpha\in \mathrm{GL}(\R^n)$, $\beta\in
\mathrm{L}(\R^n,(\R^n)^*)$,  and $\beta\alpha \in \mathrm{L_s}(\R^n,
(\R^n)^*)$, the space of symmetric linear mappings. Note that
every element of $\mathrm{Sp_v}(T^* \R^n)$ can be decomposed as
\begin{equation}
\label{fs}
B =  \left( \begin{array}{cc} \alpha^{-1} & 0 \\ \beta & \alpha^T
  \end{array} \right) = \left( \begin{array}{cc} I & 0 \\ \beta \alpha
    & I  \end{array} \right) \left( \begin{array}{cc} \alpha^{-1} & 0
    \\ 0  & \alpha^T \end{array} \right).
\end{equation}
The second result of this section is the following:

\begin{prop}
\label{vanish2}
Let $V_0,V_1$ be linear subspaces of $\R^n$, and
let $B:[0,1]\rightarrow \mathrm{Sp_v}(T^* \R^n)$ be a continuous path
such that $B(0)N^*V_0 = N^*V_0$ and $B(1)N^* V_1 = N^*V_1$. Then
\[
\mu(B N^*V_0,N^*V_1) = 0.
\]
\end{prop}

\begin{proof}
By (\ref{fs}), there are continuous paths $\alpha:[0,1] \rightarrow
\mathrm{GL}(\R^n)$ and $\gamma: [0,1] \rightarrow \mathrm{L_s}(\R^n,
(\R^n)^*)$ such that $B=GA$ with
\[
G := \left( \begin{array}{cc} I & 0 \\ \gamma & I \end{array} \right),
\quad A:=  \left( \begin{array}{cc} \alpha^{-1} & 0 \\ 0 & \alpha^T
  \end{array} \right).
\]
The assumptions on $B(0)$ and $B(1)$ and the special form of $G$ and
$A$  imply that
\begin{equation}
\label{fix}
A(0) N^* V_0 = G(0) N^* V_0 = N^* V_0, \quad
A(1) N^* V_1 = G(1) N^* V_1 = N^* V_1.
\end{equation}
The affine path $F(t) := t G(1) + (1-t) G(0)$ is homotopic with fixed
end-points to the path $G$ within the symplectic group
$\mathrm{Sp}(T^*  \R^n)$, so by the homotopy property of the Maslov index
\begin{equation}
\label{uno}
\mu(B N^* V_0,N^* V_1) = \mu(GA N^* V_0,N^* V_1) = \mu (F A N^* V_0,
N^*  V_1).
\end{equation}
We can write $F(t)$ as $F_1 (t) F_0 (t)$, where $F_0$ and $F_1$ are
the  symplectic paths
\[
F_0(t) := \left( \begin{array}{cc} I & 0 \\ (1-t) \gamma(0) & I
  \end{array}  \right), \quad
F_1(t) := \left( \begin{array}{cc} I & 0 \\ t \gamma(1) & I
  \end{array}  \right).
\]
We note that $F_0(t)$ preserves $N^* V_0$, while $F_1(t)$ preserves
$N^*  V_1$, for every $t\in [0,1]$. Then, by the naturality property
of  the Maslov index
\begin{equation}
\label{due}
\begin{split}
\mu (F A N^* V_0, N^* V_1) = \mu ( F_1 F_0 A N^* V_0, N^* V_1) \\ =
\mu( F_0 A N^* V_0, F_1^{-1} N^* V_1) = \mu( F_0 A N^* V_0, N^* V_1).
\end{split} \end{equation}
By the concatenation property of the Maslov index, (\ref{tn}),
(\ref{fix}), and the fact that $F_0(t)$ preserves $N^* V_0$ for every
$t$ and is the identity for $t=1$, we have the chain of equalities
\begin{equation}
\label{tre}
\begin{split}
\mu( F_0 A N^* V_0, N^* V_1) = \mu(F_0 A(0) N^*V_0, N^* V_1) +
\mu(F_0(1)  A N^*V_0 , N^* V_1) \\ = \mu(F_0 N^*V_0, N^* V_1) + \mu(A
N^*V_0  , N^* V_1) = \mu(N^*V_0, N^*V_1) + \mu (N^* (\alpha^{-1} V_0),
N^*V_1) = 0,
\end{split} \end{equation}
where the latter term vanishes because of Proposition \ref{vanish1}.
The conclusion follows from (\ref{uno}), (\ref{due}), and (\ref{tre}).
\qed \end{proof}

We conclude this section by discussing how graphs of symplectic
automorphisms of $T^* \R^n$ can be turned into Lagrangian subspaces of
$T^* \R^{2n}$.

Let us identify the product $T^* \R^n \times T^* \R^n$ with $T^*
\R^{2n}$.  Then the graph of the linear involution $C$ is the conormal
space of the diagonal $\Delta_{\R^n}$ in $\R^n \times \R^n$,
\[
\graf C = N^* \Delta_{\R^n}.
\]
Moreover, the fact that $C$ is anti-symplectic easily implies that
a linear endomorphism $A : T^* \R^n \rightarrow T^* \R^n$ is
symplectic if  and only if the graph of $CA$ is a Lagrangian
subspace\footnote{Here $T^* \R^{2n}$ is endowed with its standard
  symplectic  structure. In symplectic geometry it is also customary
  to  endow the product of a symplectic vector space $(V,\omega)$ with
  itself by the symplectic structure $\omega \times (-\omega)$. With
  the  latter convention, the product of two Lagrangian subspaces is
  Lagrangian, and an endomorphism is symplectic if and only if its
  graph is  Lagrangian. When dealing with cotangent spaces and
  conormal  spaces it seems more convenient to adopt the former
  convention,  even if it involves the appearance of the involution
  $C$.}  of $T^* \R^{2n}$, which is true if and only if the graph of $AC$ is a
Lagrangian  subspace of $T^* \R^{2n}$.

Theorem 3.2 in \cite{rs93} implies that if $A$ is a path of symplectic
automorphisms of $T^* \R^n$ and $\lambda$, $\nu$ are paths of
Lagrangian  subspaces of $T^* \R^n$, then
\begin{equation}
\label{prd}
\mu( A \lambda, \nu) = \mu ( \graf AC, C\lambda \times \nu) = -
\mu(\graf  CA, \lambda \times C\nu).
\end{equation}

\section{Conormal bundles}

Let $M$ be a smooth manifold\footnote{Unless otherwise stated, manifolds and submanifolds are always assumed to have no boundary.} of dimension $n$, 
and let $T^*M$ be the
cotangent bundle of $M$ with projection $\tau^* : T^*M \to M$. The
cotangent bundle $T^*M$ carries the following canonical structures:
The {\em Liouville one-form\/} $\theta$ and the {\em Liouville
vector field\/} $\eta$ which can be defined by
\[
\theta(x) [\zeta] = x\big[ D\tau^*(x)[\zeta]\big]=
d\theta (x) [\eta, \zeta],  \qquad \ \forall\, x\in T^*M, \; \forall
\zeta  \in T_xT^*M,
\]
and the symplectic structure $\omega=d\theta$. Elements of $T^* M$ are
also denoted as pairs $(q,p)$, with $q\in M$, $p\in T_q^* M$. 

The {\em vertical subbundle} is the $n$-dimensional vector subbundle
of  $TT^*M$ whose fiber at $x\in T^* M$ is the linear subspace
\[
T_x^v T^*M := \ker D\tau^* (x)\subset T_x T^*M.
\]
Each vertical subspace $T_x^v T^*M$ is a Lagrangian subspace of the
symplectic vector space $(T_x T^*M,\omega_x)$.

If $Q$ is a smooth submanifold of $M$, the {\em conormal bundle} of
$Q$ is  defined by
\[
N^* Q := \set{x\in T^*M}{\tau^* (x)\in Q, \; x[\xi] = 0 \; \forall \xi
  \in  T_{\tau^*(x)} Q}.
\]
It inherits the structure of a vector bundle over $Q$ of dimension
$\codim  \, Q$. The conormal bundle of
the  whole $M$ is the zero-section, while the conormal bundle of a
point  $Q=\{q\}$ is $T^*_q M$. Moreover, the Liouville one-form $\theta$ vanishes identically on $N^* Q$, in particular $N^* Q$ is a Lagrangian 
submanifold of $T^*M$, i.e.\ its tangent space at every point $x$ is a
Lagrangian  subspace of $(T_x T^*M,\omega_x)$. Actually, the converse is also true:

\begin{prop}
\label{conorcar}
Let $M$ be a smooth $n$-dimensional manifold and let $L$ be an $n$-dimensional submanifold of $T^*M$ on which the Liouville one-form $\theta$ vanishes identically. Then the intersection of $L$ with the zero-section of $T^*M$ is a smooth submanifold $R$, and if $L$ is a closed subset of $T^*M$ then $L=N^* R$.
\end{prop}

\begin{proof}
{\em Claim 1. The intersection $R:= L\cap M$, where $M$ denotes the zero-section of $T^*M$, is a smooth submanifold.} The matter being local, up to choosing suitable local coordinates $q_1,\dots,q_n$ and conjugate coordinates $p_1,\dots,p_n$, we may assume that $M=\R^n$ and $L$ is a graph of the form
\begin{equation}
\label{locdes}
L = \set{(q,Q(q,p),P(q,p),p)}{q \in \R^k, \; p \in (\R^k)^{\perp} },
\end{equation}
where $\R^k$ denotes the subspace of $\R^n$ spanned by the first $k$ vectors of the standard basis, $0\leq k \leq n$, $Q$ is a smooth map into $\R^{n-k}$ -- the subspace of $\R^n$ spanned by the last $n-k$ vectors of the standard basis -- and $P$ is a smooth map into $(\R^{n-k})^{\perp}$. Here we are also using the fact that $L$ is Lagrangian, so that its tangent space at a given point can be represented as in Remark \ref{linlag}. 
Then the intersection of $L$ with the zero-section is the set
\begin{equation}
\label{inters}
R = L \cap ( \R^n \times (0) ) = \set{(q,Q(q,0),0,0)}{q\in \R^k \mbox{ such that } P(q,0)=0}.
\end{equation}
The fact that the Liouville one-form $\theta$ vanishes on $L$ is equivalent to the fact that the maps $Q=(Q_{k+1},\dots,Q_n)$ and $P=(P_1,\dots,P_k)$ satisfy the identity
\[
\sum_{j=1}^k P_j(q,p) \, dq_j + \sum_{j=k+1}^n p_j \, dQ_j (q,p) = 0, \quad \forall (q,p) \in \R^k \times (\R^k)^{\perp},
\]
from which we deduce the identities
\begin{equation}
\label{tv1}
P_j(q,p) + \sum_{h=k+1}^n p_h \frac{\partial Q_h}{\partial q_j} (q,p)  = 0,  \quad \forall (q,p) \in \R^k \times (\R^k)^{\perp}, \quad \forall j=1,\dots,k.
\end{equation}
In particular, (\ref{tv1}) implies that $P(q,0)=0$ for every $q\in \R^k$, so by (\ref{inters}) the intersection of $L$ with the zero-section is
\begin{equation}
\label{desR}
R = L \cap ( \R^n \times (0) ) = \set{(q,Q(q,0),0,0)}{q\in \R^k},
\end{equation}
which is a smooth submanifold.

\medskip

\noindent{\em Claim 2. The Liouville vector field $\eta$ is tangent to $L$.} In fact, for every $x\in L$ and every $\zeta\in T_x L$ we have
\[
\omega(x) [\eta(x),\zeta] = \theta(x)[\zeta] = 0,
\]
so $\eta(x)$ is in the symplectic orthogonal space of $T_x L$. But since $L$ is a Lagrangian submanifold, such a symplectic orthogonal space is $T_x L$ itself. 

\medskip

\noindent{\em Claim 3. If, moreover, $L$ is a closed subset of $T^*M$, then $L$ is contained in $N^* R$.} Let $x$ be a point in $L$. The orbit of $x$ by the Liouville flow -- that is, the flow of the Liouville vector field $\eta$ -- converges to the point in the zero-section $(\tau^*(x),0)$ for $t\rightarrow -\infty$. By Claim 2 and by the fact that $L$ is a closed subset, we deduce that $(\tau^*(x),0)$ belongs to $L$, hence to $R$. Since both $N^* R$ and $L$ are invariant with respect to the Liouville flow, we may assume that $x$ is so close to the zero-section that it lies in the portion of $L$ which is locally described by (\ref{locdes}). Then $x$ is of the form $(q,Q(q,p),P(q,p),p)$, for some $q\in \R^k$ and $p\in (\R^k)^{\perp}$. Since the Liouville flow is 
equivariant with respect to cotangent bundle charts, $\tau^*(x)$ is the point $(q,Q(q,p))$. 
By (\ref{desR}), $\tau^*(x) = (q,Q(q,p)) = (q,Q(q,0))$, and by differentiating this identity with respect to $q_j$ we find
\[
\frac{\partial Q}{\partial q_j} (q,p) = \frac{\partial Q}{\partial q_j} (q,0), \quad \forall j=1,\dots,k.
\]
Therefore, any element $\zeta$ of $T_{\tau^*(x)} R$ is of the form  
\[
\zeta = (\xi,DQ(q,0)[\xi]) = (\xi, DQ(q,p)[\xi]),
\]
for some $\xi\in \R^k$.
Using (\ref{tv1}) again, we get
\begin{eqnarray*}
x[\zeta] = \sum_{j=1}^k P_j(q,p) \xi_j + \sum_{j=k+1}^n p_j \sum_{h=1}^k \frac{\partial Q_j}{\partial q_h} (q,p) \xi_h = \sum_{j=1}^k P_j(q,p) \xi_j + \sum_{h=1}^k \xi_h \sum_{j=k+1}^n  p_j \frac{\partial Q_j}{\partial q_h} (q,p)  \\
= \sum_{j=1}^k \left( P_j(q,p) + \sum_{h=k+1}^n  p_h \frac{\partial Q_h}{\partial q_j} (q,p)  \right) \xi_j = 0.
\end{eqnarray*}
Therefore, $x$ belongs to $N^* R$, as claimed.

\medskip

\noindent{\em Conclusion.} Since $L$ is an $n$-dimensional submanifold of the $n$-dimensional manifold $N^* R$ and it is a closed subset, it is a union of connected components of $N^* R$. But since $R$ is contained in $L$ and the conormal bundle of a connected submanifold is connected, we conclude that $L = N^* R$.
\end{proof} \qed   

\section{Hamiltonian systems on cotangent bundles with conormal
  boundary  conditions}
\label{hscbcbs}

Let $M$ be a smooth manifold of dimension $n$.
A smooth Hamiltonian $H : [0,1] \times T^*M \to \R$ induces a time
dependent vector field $X_H$  on $T^*M$ defined by
\[
\omega\big(X_H(t,x),
\zeta\big)\, = \, -D_x H(t,x)[\zeta], \quad \forall\  \zeta \in
T_xT^*M.
\]
We denote by $\phi_t^H$ the non-autonomous flow determined by the ODE
\begin{equation}
\label{ham}
x'(t) = X_H(t, x(t)).
\end{equation}

\paragraph{Local boundary conditions.}
Let $Q_0$ and $Q_1$ be submanifolds of $M$. We are interested in the
set of solutions $x:[0,1]\rightarrow T^*M$ of the Hamiltonian
system (\ref{ham}) such that
\begin{equation}
\label{locbdry}
x(0) \in N^* Q_0, \quad x(1) \in N^* Q_1.
\end{equation}
In other words, we are considering Hamiltonian orbits 
$t\mapsto (q(t),p(t))$ such that $q(0)\in Q_0$, $q(1)\in Q_1$, $p(0)$
vanishes on $T_{q(0)} Q_0$, and $p(1)$ vanishes on $T_{q(1)} Q_1$.
In particular, when $Q_0 = \{q_0\}$ and $Q_1 = \{q_1\}$ are points, we
find  Hamiltonian orbits whose projection onto $M$ joins $q_0$ and
$q_1$,  without any other conditions. When $Q_0=Q_1=M$,
(\ref{locbdry})  reduces to the Neumann boundary conditions $p(0)=p(1)=0$.

The {\em nullity} $\nu^{Q_0,Q_1}(x)$ of the solution $x$ of
(\ref{ham}-\ref{locbdry}) is the non-negative integer
\[
\nu^{Q_0,Q_1}(x) = \dim D\phi_1^H (x(0)) T_{x(0)} N^* Q_0 \cap
T_{x(1)}  N^* Q_1,
\]
and $x$ is said to be {\em non-degenerate} if $\nu^{Q_0,Q_1}(x)=0$, or
equivalently if $\phi_1^H(N^* Q_0)$ is transverse to $N^* Q_1$ at $x(1)$.

We wish to associate a Maslov index to each solution of the boundary
problem (\ref{ham}-\ref{locbdry}). If $x:[0,1]\rightarrow T^*M$ is
such a  solution, let $\Phi$ be a {\em vertical preserving symplectic
  trivialization} of the symplectic bundle $x^*(TT^*M)$. That is, for every
$t\in  [0,1]$, $\Phi(t)$ is a symplectic linear isomorphism from
$T_{x(t)}  T^*M$ to $T^* \R^n$,
\[
\Phi(t) : T_{x(t)} T^*M \rightarrow T^* \R^n,
\]
 which maps $T^v_{x(t)} T^*M$ onto the vertical subspace $N^*(0)=(0)
 \times (\R^n)^*$, and the dependence of $\Phi$ on $t$ is smooth. Moreover, we
 assume  that the tangent spaces of the conormal bundles of $Q_0$ and
 $Q_1$  are mapped into conormal subspaces of $T^* \R^n$:
\begin{equation}
\label{conincon}
\Phi(0) T_{x(0)} N^* Q_0 \in \mathscr{N}^*(\R^n) \quad \mbox{and}
\quad  \Phi(1) T_{x(1)} N^* Q_1\in \mathscr{N}^*(\R^n).
\end{equation}
Let $V_0^{\Phi}$ and $V_1^{\Phi}$ be the linear subspaces of $\R^n$
defined  by
\[
N^* V_0^{\Phi} =   \Phi(0) T_{x(0)} N^* Q_0, \quad
N^* V_1^{\Phi} =   \Phi(1) T_{x(1)} N^* Q_1.
\]
The fact that $\Phi$ maps the vertical subbundle into the vertical
subspace  implies that $\dim V_0^{\Phi} = \dim Q_0$
and $\dim V_1^{\Phi} = \dim Q_1$. Since the flow $\phi_t^H$ is
symplectic,  the linear mapping
\begin{equation}
\label{g}
G^{\Phi}(t):= \Phi(t) D \phi_t^H (x(0)) \Phi(0)^{-1}
\end{equation}
is a symplectic automorphism of $T^*\R^n$. Notice that
\[
\nu^Q(x) = \dim G^{\Phi} (1) N^* V_0^{\Phi} \cap N^* V_1^{\Phi}.
\]

\begin{defn}
The Maslov index of a solution $x$ of (\ref{ham}-\ref{locbdry}) is the
half-integer
\[
\mu^{Q_0,Q_1}(x) := \mu( G^{\Phi} N^* V_0^{\Phi}, N^* V_1^{\Phi}) +
 \frac{1}{2} (\dim Q_0 + \dim Q_1 - n).
\]
\end{defn}
The shift
\[
\frac{1}{2} (\dim Q_0 + \dim Q_1 - n)
\]
comes from the fact that in the case of convex Hamiltonians we would
like  the Maslov index of a non-degenerate solution to coincide with
the  Morse index of the corresponding extremal curve of the Lagrangian
action functional (see Section \ref{morindsec} below).
The next result shows that the Maslov index of $x$ is well-defined:

\begin{prop}
\label{invdef}
Assume that $\Phi$ and $\Psi$ are two vertical preserving symplectic
trivializations of $x^*(TT^*M)$, and that they both satisfy        
(\ref{conincon}). Then
\begin{equation}
\label{indep}
\mu( G^{\Phi} N^* V_0^{\Phi}, N^* V_1^{\Phi}) = \mu( G^{\Psi} N^*  
V_0^{\Psi}, N^* V_1^{\Psi}).
\end{equation}
If $x$ is non-degenerate, the number $\mu^{Q_0,Q_1}(x)$ is an integer.
\end{prop}

\begin{proof}
Since both $\Phi$ and $\Psi$ are vertical preserving, the path $B(t):=
\Psi(t) \Phi(t)^{-1}$ takes values into the subgroup
$\mathrm{Sp_v}(T^*  \R^n)$. We first prove the identity (\ref{indep})
under the extra assumption
\begin{equation}
\label{coinc}
V_0^{\Phi} = V_0^{\Psi} = V_0, \quad V_1^{\Phi} = V_1^{\Psi} = V_1.
\end{equation}
In this case,
\begin{equation}
\label{pres}
B(0) N^* V_0 = N^* V_0, \quad B(1) N^* V_1 = N^* V_1.
\end{equation}
Consider the homotopy of Lagrangian subspaces
\[
\lambda(s,t) := B(s) G^{\Phi}(st) N^* V_0.
\]
By the concatenation and the homotopy property of the Maslov index,
\begin{equation}
\label{conhom}
\mu(\lambda|_{[0,1] \times \{0\}},N^* V_1) + \mu(\lambda|_{\{1\}
  \times  [0,1]},N^* V_1) = \mu(\lambda|_{\{0\} \times [0,1]},N^* V_1)
  +  \mu(\lambda|_{[0,1] \times \{1\}},N^* V_1).
\end{equation}
Since $\lambda(0,t) = B(0) N^* V_0$ is constant in $t$,
\begin{equation}
\label{l1}
\mu(\lambda|_{\{0\} \times [0,1]}, N^* V_1) = 0.
\end{equation}
By the naturality of the Maslov index and since $B(1)$ preserves $N^* V_1$,
\begin{equation}
\label{l2}
\mu(\lambda|_{\{1\} \times [0,1]},N^* V_1) = \mu (B(1) G^{\Phi} N^*
V_0,  N^* V_1) = \mu ( G^{\Phi} N^* V_0, N^* V_1).
\end{equation}
Moreover,
\begin{equation}
\label{l3}
\mu(\lambda|_{[0,1]\times \{0\}}, N^* V_1) = \mu(B N^* V_0,  N^* V_1) = 0,
\end{equation}
because of (\ref{pres}) and Proposition \ref{vanish2}. Finally,
\begin{equation}
\label{l4}
\mu(\lambda|_{[0,1]\times \{1\}}, N^* V_1) = \mu ( B G^{\Phi} N^* V_0,
N^* V_1) = \mu ( G^{\Psi}  N^* V_0, N^* V_1).
\end{equation}
Then (\ref{conhom}) together with (\ref{l1}), (\ref{l2}), (\ref{l3}),
and  (\ref{l4}) imply the identity (\ref{indep}) under the extra
assumption (\ref{coinc}). 

Now we deal with the general case. Let $\alpha_0,\alpha_1: [0,1]  
\rightarrow \mathrm{GL}(\R^n)$ be continuous paths such that
\[
\alpha_0(1)= \alpha_1(0) = I, \quad \alpha_0(0) V_0^{\Psi} =
V_0^{\Phi},  \quad \alpha_1(1) V_1^{\Psi} = V_1^{\Phi}.
\]
Consider the paths in $\mathrm{Sp_v}(T^* \R^n)$
\[
A_0 = \left( \begin{array}{cc} \alpha_0^{-1} & 0 \\ 0 & \alpha_0^T
  \end{array} \right), \quad 
A_1 = \left( \begin{array}{cc} \alpha_1^{-1} & 0 \\ 0 & \alpha_1^T   
\end{array} \right).
\]
Then $A_0(1)=A_1(0)=I$, and by (\ref{tn})
\[
A_0(0) N^* V_0^{\Phi} = N^* V_0^{\Psi}, \quad
A_1(1) N^* V_1^{\Phi} = N^* V_1^{\Psi}.
\]
The trivialization $\Theta(t):= A_1(t) A_0(t) \Phi(t)$ is vertical
preserving,  and
\begin{eqnarray*}
\Theta(0) T_{x(0)} N^* Q_0 = A_1(0) A_0(0) \Phi(0) T_{x(0)} N^* Q_0 =
A_0(0)  N^* V_0^{\Phi} = N^* V_0^{\Psi}, \\
\Theta(1) T_{x(1)} N^* Q_1 = A_1(1) A_0(1) \Phi(1) T_{x(1)} N^* Q_1 =
A_1(1)  N^* V_1^{\Phi} = N^* V_1^{\Psi}.
\end{eqnarray*}
Therefore, $\Theta$ is an admissible trivialization with $V_0^{\Theta}
=  V_0^{\Psi}$ and $V_1^{\Theta} = V_1^{\Psi}$. By the particular case
treated above,
\[
\mu (G^{\Theta} N^*V_0^{\Theta}, N^* V_1^{\Theta}) = \mu (G^{\Psi}
N^*V_0^{\Psi}, N^* V_1^{\Psi}),
\]
so it is enough to prove that the left-hand side coincides with 
$\mu (G^{\Phi} N^*V_0^{\Phi}, N^* V_1^{\Phi})$. By the naturality
property  of the Maslov index,
\[
\mu (G^{\Theta} N^*V_0^{\Theta}, N^* V_1^{\Theta}) = \mu ( A_1 A_0
G^{\Phi}  N^*V_0^{\Phi}, A_1(1) N^* V_1^{\Phi}) =  \mu ( A_1(1)^{-1}
A_1 A_0  G^{\Phi} N^*V_0^{\Phi}, N^* V_1^{\Phi}).
\]
By the concatenation property of the Maslov index, the latter quantity
coincides with
\begin{eqnarray*}
\mu( A_1(1)^{-1} A_1 A_0 G^{\Phi}(0) N^* V_0^{\Phi}, N^* V_1^{\Phi}) +
\mu ( A_1(1)^{-1} A_1(1) A_0(1) G^{\Phi} N^* V_0^{\Phi}, N^*
V_1^{\Phi}) \\  = \mu( A_1(1)^{-1} A_1 A_0 N^* V_0^{\Phi}, N^*
V_1^{\Phi})  + \mu (G^{\Phi} N^* V_0^{\Phi}, N^* V_1^{\Phi}).
\end{eqnarray*}
By (\ref{tn}), $A_1(1)^{-1} A_1(t) A_0(t) N^* V_0^{\Phi}$ is a
conormal  subspace of $T^* \R^n$ for every $t\in [0,1]$, so the first
term after the equal sign in the expression above vanishes because of Proposition
\ref{vanish1}. The identity (\ref{indep}) follows.

If $x$ is non-degenerate, $G^{\Phi}(1) V_0^{\Phi} \cap V_1^{\Phi} =
(0)$, whereas the intersection $G^{\Phi}(0) V_0^{\Phi} \cap V_1^{\Phi}
= V_0^{\Phi} \cap V_1^{\Phi}$ might be non-trivial. By Corollary 4.12
in \cite{rs93}, the relative Maslov index $\mu(G^{\Phi} V_0^{\Phi},
 V_1^{\Phi})$ differs by an integer from the number $d/2$, where
\[
d := \dim N^* V_0^{\Phi} \cap N^* V_1^{\Phi}.
\]
Since
\[
N^* V_0^{\Phi} \cap N^* V_1^{\Phi} = (V_0^{\Phi} \cap V_1^{\Phi})
\times  ({V_0^{\Phi}}^{\perp} \cap {V_1^{\Phi}}^{\perp}) = (V_0^{\Phi}
\cap  V_1^{\Phi}) \times (V_0^{\Phi} + V_1^{\Phi})^{\perp},
\]
the  number
\[
d = \dim V_0^{\Phi} \cap V_1^{\Phi} + n - \dim (V_0^{\Phi} +
V_1^{\Phi}) =  \dim V_0^{\Phi} + \dim V_1^{\Phi} + n - 2 \dim
(V_0^{\Phi} +  V_1^{\Phi})
\]
has the parity of
\[
\dim Q_0 + \dim Q_1 - n = \dim V_0^{\Phi} + \dim V_1^{\Phi} - n.
\]
It follows that
\[
\mu^{Q_0,Q_1}(x) = \mu(G^{\Phi} V_0^{\Phi}, V_1^{\Phi}) + \frac{1}{2}
( \dim  Q_0 + \dim Q_1 - n)
\]
is an integer, as claimed.
\qed \end{proof}

\paragraph{Non-local boundary conditions.} The smooth involution
\[
\mathscr{C}: T^*M \rightarrow T^*M, \quad \mathscr{C}(x) = -x,
\]
is anti-symplectic, meaning that $\mathscr{C}^* \omega = -
\omega$. Its  graph in $T^*M \times T^*M = T^* M^2$ is the conormal
bundle  of the diagonal $\Delta_M$ of $M\times M$. Note also that
conormal  subbundles in $T^*M$ are $\mathscr{C}$-invariant.

Given a smooth submanifold $Q \subset M \times M$, we are interested
in the set of all solutions $x : [0,1] \to T^*M$ of (\ref{ham})
satisfying the nonlocal boundary condition
\begin{equation}
\label{nonlocbdry}
\big( x(0), - x(1)\big)\,  \in \, N^*Q.
\end{equation}
Equivalently, we are looking at the Lagrangian intersection problem
\[
\bigl( \graf \mathscr{C} \circ \phi_1^H \bigr) \cap N^* Q
\]
in $T^* M^2$. A solution $x$ of (\ref{ham}-\ref{nonlocbdry}) is called
{\em  non-degenerate} if the above intersection is transverse at
$(x(0), -x(1))$, or equivalently if the {\em nullity} of $x$, defined as
\[
\nu^Q(x) := \dim \bigl( T_{(x(0),-x(1))} \graf \mathscr{C} \circ
\phi_1^H  \bigr) \cap T_{(x(0),-x(1))} N^*Q,
\]
is zero.

When $Q=Q_0 \times Q_1$ is the product of two submanifolds $Q_0$,
$Q_1$ of  $M$, the boundary condition (\ref{nonlocbdry}) reduces to
the  local boundary condition (\ref{locbdry}).  A common choice for
$Q$ is  the diagonal $\Delta_M$ in $M\times M$: this choice produces
1-periodic Hamiltonian orbits (provided that $H$ can be extended to
$\R \times T^* M$ as a $1$-periodic function).  Other choices are also
interesting: for instance in \cite{as08} it is shown that the
pair-of-pants  product on Floer homology for periodic orbits on the
cotangent bundle of $M$ factors through a Floer homology for
Hamiltonian  orbits on $T^* (M\times M)$ with nonlocal boundary
condition  (\ref{nonlocbdry}) given by the submanifold $Q$ of
$M\times M \times M \times M$ consisting of all 4-uples $(q,q,q,q)$.

The nonlocal boundary value problem (\ref{ham}-\ref{nonlocbdry}) on
$T^*M$  can be turned into a local boundary value problem on $T^* M^2
=  T^*M \times T^*M$. Indeed, a curve $x: [0,1] \rightarrow T^*M$ is
an  orbit for the Hamiltonian vector field $X_H$ on $T^*M$ if and only
if  the curve
\[
y: [0,1] \rightarrow T^* M^2, \quad y(t) = \bigl( x(t/2), -x(1-t/2)\bigr),
\]
is an orbit for the Hamiltonian vector field $X_K$ on $T^*M^2$, where
$K\in  C^{\infty}([0,1]\times T^*M^2)$ is the Hamiltonian
\[
K(t,y_1,y_2) := \frac{1}{2} H(t/2,y_1) + \frac{1}{2} H(1-t/2,-y_2).
\]
By construction,
\[
y(1) = \bigl( x(1/2), - x(1/2)) \in \graf \mathscr{C} = N^* \Delta_M,
\]
and the curve $x$ satisfies the nonlocal boundary condition
(\ref{nonlocbdry})  if and only if
\[
y(0) = \bigl( x(0), - x(1)) \in N^* Q.
\]
Therefore, the nonlocal boundary value problem
(\ref{ham}-\ref{nonlocbdry})  for $x:[0,1] \rightarrow T^*M$ is
equivalent  to the following local boundary value problem for $y:
[0,1]  \rightarrow T^*M^2$:
\begin{eqnarray}
\label{kham}
& y'(t) = X_K(t,y(t)), &\\
\label{nbc}
& y(0) \in N^* Q, \quad y(1) \in N^* \Delta_M. &
\end{eqnarray}
Using the identity (\ref{flusso}) below, it is easy to show that
\[
\nu^{Q,\Delta_M}(y) = \nu^Q (x).
\]
In particular, $x$ is a non-degenerate solution of (\ref{ham}-  
\ref{nonlocbdry}) if and only if $y$ is non-degenerate solution of 
(\ref{kham}-\ref{nbc}). We define the Maslov index of the solution $x$
of  (\ref{ham}-\ref{nonlocbdry}) as the Maslov index of the solution
$y$ of  (\ref{kham}-\ref{nbc}):
\[
\mu^Q(x) := \mu^{Q,\Delta_M} (y).
\]
It is also convenient to have a formula for the latter Maslov index
which  avoids the above local reformulation.

\begin{prop}
\label{locfor}
Assume that $\Phi$ is a vertical preserving symplectic trivialization of
$x^*(TT^*M)$, and that the linear subspace
\[
\bigl( \Phi(0) \times C \Phi(1) D\mathscr{C}(-x(1)) \bigr)
T_{(x(0),-x(1))}  N^* Q \subset T^* \R^n \times T^* \R^n = T^* \R^{2n}
\]
is a conormal subspace of $T^* \R^{2n}$, that we denote by $N^*
W^{\Phi}$,  with $W^{\Phi} \in \mathrm{Gr}(\R^{2n})$. Then
\begin{equation}
\label{nonlocfor}
\mu^Q(x) = \mu (\graf G^{\Phi} C, N^* W^{\Phi}) + \frac{1}{2} ( \dim Q - n),
\end{equation}
where $G^{\Phi}$ is defined by (\ref{g}). In particular, if $Q=Q_0
\times  Q_1$ with $Q_0$ and $Q_1$ smooth submanifolds of $M$, then
$\mu^Q(x)  = \mu^{Q_0,Q_1}(x)$.
\end{prop}

\begin{proof}
The isomorphisms
\[
\Psi(t) := \Phi(t/2) \times C \Phi(1-t/2) D\mathscr{C}(-x(1-t/2)) :
T_{y(t)}  T^* M^2 \rightarrow T^* \R^{2n}
\]
provide us with a vertical preserving symplectic trivialization of
$y^*(TT^* M^2)$. By assumption,
\[
\Psi(0) T_{y(0)} N^* Q = \bigl( \Phi(0) \times C \Phi(1)
D\mathscr{C}(-x(1))  \bigr)  T_{(x(0),-x(1))} N^* Q = N^* W^{\Phi}.
\]
Moreover, since $\mathscr{C}$ is an involution,
\begin{equation}
\label{diag}
\begin{split}
\Psi(1) T_{y(1)} N^* \Delta_M = \Psi(1) T_{(x(1/2),-x(1/2))} \graf
\mathscr{C} = \Psi(1) \graf D\mathscr{C}(x(1/2)) \\ =  \bigl(
\Phi(1/2)  \times C \Phi(1/2) D\mathscr{C}(-x(1/2)) \bigr)  \graf D
\mathscr{C}(x(1/2)) \\ = \set{(\Phi(1/2) \xi, C \Phi(1/2) \xi)}{\xi
  \in  T_{x(1/2)} T^*M}  =
\graf C = N^* \Delta_{\R^n}.
\end{split} \end{equation}
Therefore, $\Psi$ is an admissible trivialization of $y^*(TT^*M^2)$, and
\begin{equation}
\label{launo}
\mu^Q(x) = \mu^{Q,\Delta_M}(y) = \mu (G^{\Psi} N^* W^{\Phi}, N^*
\Delta_{\R^n} ) \\ + \frac{1}{2} (\dim Q + \dim \Delta_M - 2n),
\end{equation}
where $G^{\Psi}$ is defined as in (\ref{g}) by
\[
G^{\Psi}(t) := \Psi(t) D\phi_t^K(y(0)) \Psi(0)^{-1}.
\]
We denote by $\phi^H_{t,s}$ the solution of
\[
\left\{
\begin{array}{rcl} \phi^H_{s,s} (z) & = & z, \\
\partial_t \phi^H_{t,s} (z) & = &X_H(t,\phi^H_{t,s}(z)),
\end{array} \right. \quad \forall z\in T^*M, \; \forall s,t\in [0,1],
\]
and we omit the second subscript $s$ when $s=0$. By differentiating
the  identity
\[
\phi_{r,t}^H(\phi_{t,s}^H(z)) = \phi_{r,s}^H(z)\quad \forall z\in T^* M, \;
\forall r,s,t\in [0,1],
\]
we find
\begin{equation}
\label{cocycle}
D \phi_{r,t}^H (\phi_{t,s}^H(z)) D \phi_{t,s}^H (z) = D \phi_{r,s}^H (z) \quad
\forall z\in T^* M, \; \forall r,s,t\in [0,1].
\end{equation}
By construction, the flow of $X^K$ is related to the flow of $X^H$ by
the  formula
\[
\phi^K_t(y_1,y_2) = \bigl( \phi_{t/2}^H(y_1), - \phi_{1-t/2,1}^H(-y_2) \bigr).
\]
It follows that
\begin{equation}
\label{flusso}
D \phi_t^K(y(0)) = D \phi_{t/2}^H(x(0)) \times D \mathscr{C}(x(1-t/2))
D  \phi_{1-t/2,1}^H (x(1)) D\mathscr{C}(-x(1)),
\end{equation}
and
\[
G^{\Psi}(t) = \Phi(t/2) D \phi_{t/2}^H(x(0)) \Phi(0)^{-1} \times C
\Phi(1- t/2) D \phi^H_{1-t/2,1} (x(1)) \Phi(1)^{-1} C.
\]
By (\ref{cocycle}), the inverse of this isomorphism can be written as
\[
G^{\Psi}(t)^{-1} = \Phi(0) D \phi_{t/2}^H(x(0))^{-1} \Phi(t/2)^{-1}
\times C  \Phi(1) D \phi_{1,1-t/2}^H (x(1-t/2)) \Phi(1-t/2)^{-1} C.
\]
Then
\[
G^{\Psi}(t)^{-1} N^* \Delta_{\R^n} = G^{\Psi}(t)^{-1} \graf C \\ =
\graf C A(t),
\]
where $A$ is the symplectic path
\[
A(t):= \Phi(1) D \phi_{1,1-t/2}^H (x(1-t/2)) \Phi(1-t/2)^{-1}
\Phi(t/2) D  \phi_{t/2}^H(x(0)) \Phi(0)^{-1}.
\]
Note that
\[
A(0) = I, \quad A(1) = \Phi(1) D \phi_1^H(x(0)) \Phi(0)^{-1},
\]
and that this path is homotopic with fixed end-points to the path
$G^{\Phi}$  by the symplectic homotopy mapping $(s,t)$ into
\[
\Phi(1) D \phi_{1,1-\frac{t}{2}- s\frac{1-t}{2}}^H
\bigl(x(1-\frac{t}{2}- s \frac{1-t}{2})\bigr) \Phi\bigl(1 -
\frac{t}{2} - s \frac{1-t}{2}\bigr)^{-1} \Phi\bigl(\frac{t}{2} + s
\frac{1-t}{2}\bigr)  D \phi_{\frac{t}{2} + s\frac{1-t}{2}}^H(x(0))
\Phi(0)^{-1}.
\]
Therefore, by the naturality and the homotopy properties of the Maslov index,
\begin{equation}
\label{ladue}
\begin{split}
\mu (G^{\Psi} N^* W^{\Phi}, N^* \Delta_{\R^n}) = \mu (N^* W^{\Phi}, 
{G^{\Psi}}^{-1} N^* \Delta_{\R^n}) = \mu( N^* W^{\Phi}, \graf CA) \\ = 
\mu (N^*W^{\Phi},\graf C G^{\Phi}) = - \mu (\graf C G^{\Phi}, N^*W^{\Phi}).
\end{split}
\end{equation}
The conormal subspace $N^* W^{\Phi}$ is invariant with respect to the
anti-symplectic involution $C\times C$, while $(C \times C) \graf
CG^{\Phi}  = \graf G^{\Phi} C$. Then the identity (\ref{clag}) implies that
\begin{equation}
\label{latre}
\mu (\graf C G^{\Phi}, N^*W^{\Phi}) = - \mu (\graf G^{\Phi} C, N^* W^{\Phi}).
\end{equation}
Formulas (\ref{ladue}) and (\ref{latre}) imply
\begin{equation}
\label{laquattro}
\mu (G^{\Psi} N^* W^{\Phi} , N^* \Delta_{\R^n}) = \mu (\graf
G^{\Phi}C, N^*  W^{\Phi}).
\end{equation}
Identities  (\ref{launo}) and (\ref{laquattro}) imply (\ref{nonlocfor}).

If $Q=Q_0\times Q_1$, we have $N^* Q = N^* Q_0 \times N^* Q_1$,
and we  can choose a vertical preserving symplectic trivialization
$\Phi$ of  $x^*(TT^*M)$ such that
\[
\Phi(0) T_{x(0)} N^* Q_0 = N^* V_0^{\Phi}, \quad \Phi(1) T_{x(1)} N^*
Q_1 =  N^* V_1^{\Phi},
\]
with $V_0^{\Phi}$ and $V_1^{\Phi}$ in $\mathrm{Gr}(\R^n)$. It follows
that  $W^{\Phi} = V_0^{\Phi} \times V_1^{\Phi}$, and by the identity
(\ref{prd}) we have
\begin{equation}
\label{lacinque}
\mu ( \graf G^{\Phi} C, N^* W^{\Phi} ) = \mu (\graf G^{\Phi} C, N^* 
V_0^{\Phi} \times N^* V_1^{\Phi}) = \mu (G^{\Phi} N^* V_0^{\Phi}, N^* 
V_1^{\Phi}).
\end{equation}
By (\ref{nonlocfor}) and (\ref{lacinque}) we deduce that
\[
\mu^Q (x) = \mu (G^{\Phi} N^* V_0^{\Phi}, N^* V_1^{\Phi}) + \frac{1}{2}
(\dim Q_0 + \dim Q_1 - n),
\]
which is precisely $\mu^{Q_0,Q_1}(x)$. This concludes the proof.
\qed \end{proof}

\begin{rem} {\em (Periodic boundary conditions)}
Let us consider the particular case $Q=\Delta_M$ and $H$ 1-periodic in time, so that $x=(q,p)$ is a 1-periodic orbit. If the vector bundle $q^* (TM)$ over the circle $\R/\Z$ is orientable - hence trivial - there are vertical preserving trivializations of $x^*(TT^*M)$ which are 1-periodic in time. If $\Phi$ is such a trivialization, we have
\[
\bigl( \Phi(0) \times C \Phi(1) D\mathscr{C}(-x(1)) \bigr)
T_{(x(0),-x(1))}  N^* \Delta_M = N^* \Delta_{\R^n},
\]
see (\ref{diag}). So the trivialization $\Phi$ satisfies the assumption of Proposition \ref{locfor}, and the  Maslov index of the periodic orbit $x$ is
\[
\mu^{\Delta_M}(x) = \mu (\graf G^{\Phi} C, N^* \Delta_{\R^n}),
\]
which is precisely the Conley-Zehnder index $\mu_{CZ}(G^{\Phi})$ of
the  symplectic path $G^{\Phi}$. If $M$ is not orientable and $x=(q,p)$ is a
1-periodic orbit such that the vector bundle $q^*(TM)$ over $\R/\Z$ is not  
orientable, then there are no vertical preserving periodic
trivializations  of $x^*(TT^*M)$. In this case, any trivialization of $q^*(TM)$ over $[0,1]$ induces a non-periodic vertical preserving trivialization of $x^*(TT^*M)$ which satisfies the assumption of Proposition \ref{locfor}, and the Maslov index $\mu^{\Delta_M}(x)$ is still given by formula (\ref{nonlocfor}). 
Alternatively, one can identify
suitable  classes of non-vertical-preserving periodic trivializations
for  which the formula relating $\mu^{\Delta_M}(x)$ to the
Conley-Zehnder  index of the corresponding symplectic path involves
just a  correction term $+1$, as in \cite{web02}.
\end{rem}

\section{The dual Lagrangian formulation and the index theorem}
\label{morindsec}

In this section we assume that the Hamiltonian $H\in
C^{\infty}([0,1]\times  T^*M)$ satisfies the classical Tonelli
assumptions:  It is $C^2$-strictly convex and superlinear, that is,
\begin{eqnarray}
\label{convex}
D_{pp} H(t,q,p) > 0 \quad \forall (t,q,p)\in [0,1]\times T^*M, \\
\label{superlin}
\lim_{|p|\rightarrow \infty} \frac{H(t,q,p)}{|p|} = +\infty \quad 
\mbox{uniformly in } (t,q)\in [0,1]\times M.
\end{eqnarray}
Here the norm $|p|$ of the covector $p\in T_q^*M$ is induced by some
fixed  Riemannian metric on $M$. If $M$ is compact, the superlinearity
condition does not depend on the choice of such a metric.

Under these assumptions, the Fenchel transform defines a smooth, 
time-dependent Lagrangian on $TM$,
\[
L(t,q,v) := \max_{p\in T_q^*M} \bigl( \langle p,v \rangle - H(t,q,p)
\bigr), \quad \forall (t,q,v)\in [0,1]\times TM,
\]
which is also $C^2$-strictly convex and superlinear,
\begin{eqnarray*}
D_{vv} L(t,q,v) > 0 \quad \forall (t,q,v)\in [0,1]\times TM, \\
\lim_{|v|\rightarrow \infty} \frac{L(t,q,v)}{|v|} = +\infty \quad 
\mbox{uniformly in } (t,q)\in [0,1]\times M.
\end{eqnarray*}
Since the Fenchel transform is an involution, we also have
\begin{equation}
\label{fenchel}
H(t,q,p) = \max_{v\in T_q M} \bigl( \langle p,v \rangle - L(t,q,v)
\bigr), \quad \forall (t,q,p)\in [0,1]\times T^*M.
\end{equation}
Furthermore, the Legendre duality defines a diffeomorphism
\[
\mathcal{L} : [0,1] \times TM \rightarrow [0,1] \times T^*M, \quad
(t,q,v)  \rightarrow \bigl(t,q,D_v L(t,q,v)\bigl),
\]
such that
\begin{equation}
\label{equality}
L(t,q,v) = \langle p, v \rangle - H(t,q,p) \quad \iff \quad (t,q,p) = 
\mathcal{L}(t,q,v).
\end{equation}
A smooth curve $x:[0,1] \rightarrow T^*M$ is an orbit of the
Hamiltonian  vector field $X_H$ if and only if the curve $\gamma:=
\pi\circ  x:[0,1]\rightarrow M$ is an absolutely continuous extremal
of the  Lagrangian action functional
\[
\mathbb{S}_L(\gamma) := \int_0^1
L(t,\gamma(t),\gamma'(t))\, dt.
\]
The corresponding Euler-Lagrange equation can be written in local
coordinates  as
\begin{equation}
\label{euler}
\frac{d}{dt} \partial_v L(t,\gamma(t),\gamma'(t)) = \partial_q
L(t,\gamma(t), \gamma'(t)).
\end{equation}
If $Q$ is a non-empty submanifold of $M\times M$, the non-local
boundary  condition (\ref{nonlocbdry}) is translated into the conditions
\begin{eqnarray}
\label{lagnonloc1} (\gamma(0),\gamma(1)) & \in & Q, \\ \label{lagnonloc2}
D_v L(0,\gamma(0),\gamma'(0)) [\xi_0] & =  & D_v
L(1,\gamma(1),\gamma'(1))  [\xi_1] \quad \forall (\xi_0,\xi_1) \in
T_{(\gamma( 0),\gamma(1))} Q.
\end{eqnarray}
The second condition is the natural boundary condition induced by the
first  one, meaning that every curve which is an extremal curve of
$\mathbb{S}_L$ among all curves satisfying (\ref{lagnonloc1})
necessarily  satisfies (\ref{lagnonloc2}).

In order to study the second variation of $\mathbb{S}_L$ at the
extremal  curve $\gamma$, it is convenient to localize the problem in
$\R^n$.  This can be done by choosing a smooth local coordinate system
\[
[0,1] \times \R^n \rightarrow [0,1] \times M, \quad (t,q) \mapsto
(t,\varphi_t (q)),
\]
such that $\gamma(t)\in \varphi_t(\R^n)$ for every $t\in [0,1]$. Such
a  diffeomorphism induces the coordinate systems on the tangent and 
cotangent bundles given by
\begin{eqnarray}
\label{tbcs}
[0,1] \times T \R^n & \rightarrow & [0,1] \times TM, \quad (t,q,v)
\mapsto (t, \varphi_t(q),D\varphi_t(q)[v]\bigl), \\ \label{cbcs}
[0,1] \times T^* \R^n & \rightarrow & [0,1] \times T^*M \quad (t,q,p)
\mapsto  (t,\varphi_t(q),(D\varphi_t(q)^*)^{-1} [p] \bigr).
\end{eqnarray}
If we pull back the Lagrangian $L$ and the Hamiltonian $H$ by the
above  diffeomorphisms, we obtain a smooth Lagrangian on $[0,1]\times
T\R^n$  -- that we still denote by $L$ -- and a smooth Hamiltonian on
$[0,1] \times T^* \R^n$ -- that we still denote by $H$. These new
functions  are still related by Fenchel duality. The submanifold
$Q\subset  M \times M$ can also be pulled back in $\R^n \times \R^n$
by the  map $\varphi_0\times \varphi_1$. The resulting submanifold of
$\R^n  \times \R^n$ is still denoted by $Q$. The cotangent bundle
coordinate  system (\ref{cbcs}) induces a symplectic trivialization of
$x^*( TT^*M)$ which preserves the vertical subspaces and maps conormal
subbundles into conormal subbundles. In particular, this
trivialization  satisfies the assumptions of Proposition \ref{locfor}.

The solution $\gamma$ of
(\ref{euler}-\ref{lagnonloc1}-\ref{lagnonloc2}) is  now a curve
$\gamma:  [0,1] \rightarrow \R^n$. Let $i^Q(\gamma)$ be its Morse
index,  that is, the dimension of a maximal subspace of the Hilbert space
\[
W^{1,2}_W([0,1],\R^n) := \set{u\in W^{1,2}([0,1],\R^n)}{(u(0),u(1))
  \in W},  \mbox{where } W:= T_{(\gamma(0),\gamma(1))} Q,
\]
on which the second variation
\begin{eqnarray*}
d^2 \mathbb{S}_L(\gamma)[u,v] := \int_0^1 \Bigl( D_{vv}
L(t,\gamma,\gamma')  [u',u'] + D_{qv} L(t,\gamma,\gamma')[u',v] \\
+ D_{vq} L(t,\gamma,\gamma')[u,v'] + D_{qq} L(t,\gamma,\gamma')[u,v]
\Bigr)\, dt
\end{eqnarray*}
is negative definite. The nullity of such a quadratic form is denoted
by  $\nu^Q(\gamma)$,
\[
\nu^Q(\gamma) := \dim \set{u\in W^{1,2}_W([0,1],\R^n)}{d^2
  \mathbb{S}_L (\gamma)[u,v] = 0 \mbox{ for every } v\in
  W^{1,2}_W([0,1], \R^n)}.
\]
The following index theorem relates the Morse index and nullity of
$\gamma$  to the relative Maslov index and nullity of the
corresponding  Hamiltonian orbit:

\begin{thm}
\label{indfor}
Let $\gamma:[0,1] \rightarrow \R^n$ be a solution of  
(\ref{euler}-\ref{lagnonloc1}-\ref{lagnonloc2}), and let
$x:[0,1]\rightarrow  T^* \R^n$ be the corresponding Hamiltonian
orbit. Let  $\lambda$ be the path of Lagrangian subspaces of $T^* \R^n
\times T^* \R^n = T^* \R^{2n}$ defined by
\[
\lambda(t) :=  \graf D \phi_t^H(x(0)) C, \quad t\in [0,1],
\]
where $\phi_t^H$ denotes the Hamiltonian flow and $C$ is the
anti-symplectic  involution $C(q,p)=(q,-p)$. Let $W = T_{(\gamma(0), 
\gamma(1))} Q \in \mathrm{Gr} (\R^n \times \R^n)$. Then
\begin{eqnarray*}
& \nu^Q(\gamma) = \dim \lambda(1) \cap N^* W, & \\
& i^Q(\gamma) = \mu (\lambda, N^* W) + \frac{1}{2} (\dim Q - n) -  
\frac{1}{2} \nu^Q(\gamma). &
\end{eqnarray*}
\end{thm}

This theorem is essentially due to Duistermaat, see Theorem 4.3 in
\cite{dui76}. However, in Duistermaat's formulation the Morse index of
$\gamma$ is related to an absolute Maslov-type index $i(\lambda)$ of
the Lagrangian path $\lambda$ (see Definition 2.3 in
\cite{dui76}). This  choice makes the index formula more
complicated. The  use of the relative Maslov index
$\mu(\lambda,\cdot)$  introduced by Robbin and Salamon in \cite{rs93}
simplifies such a formula. Rather than deducing Theorem \ref{indfor}
from  Duistermaat's statement, we prefer to present a modified version of his proof, using
the  relative Maslov index $\mu$ instead of the absolute Maslov-type
index  $i$.

\medskip

\begin{proof}
Let  $c$ be a real number, chosen to be so large that the bilinear
form $d^2 \mathbb{S}_{L+c|q|^2}(\gamma)$ is positive definite, which is therefore a
Hilbert product on $W^{1,2}_W([0,1],\R^n)$. We denote by $\mathscr{E}$
the bounded self-adjoint operator on $W^{1,2}_W([0,1],\R^n)$ which
represents the symmetric bilinear form $d^2 \mathbb{S}_L(\gamma)$ with
respect to such a Hilbert product. It is a compact perturbation of the
identity, and $i^Q(\gamma)$ is the number of its negative eigenvalues,
counted with multiplicity (see Lemma 1.1 in \cite{dui76}), while
$\nu^Q(\gamma)$ is the dimension of its kernel. The eigenvalue
equation $\mathscr{E} u = \lambda u$ corresponds to a second order
Sturm-Liouville boundary value problem in $\R^n$. Legendre duality
shows that such a linear second order problem is equivalent to the
following  first order linear Hamiltonian boundary value problem on  
$T^* \R^n$:
\begin{equation}
\label{linham}
\xi'(t) = A(\mu,t) \xi(t), \quad (\xi(0),C\xi(1)) \in N^* W.
\end{equation}
Here
\[
A(\mu,t) := \left( \begin{array}{cc} D_{qp} H(t,x(t)) & D_{pp} H(t,x(t)) \\ -
\mu c  T - D_{qq} H(t,x(t)) & - D_{pq} H(t,x(t)) \end{array} \right),
\]
where $\mu=\lambda/(1-\lambda)$ and $T:\R^n \rightarrow (\R^n)^*$ is the isomorphism induced by the Euclidean inner product. The fact that $d^2 \mathbb{S}_{L+c|q|^2}(\gamma)$ is positive definite implies that problem (\ref{linham}) has only the zero solution when $\mu \leq -1$. Let $\Phi(\mu,t)$ be the solution of
\[
\frac{\partial \Phi}{\partial t} (\mu,t) = A(\mu,t) \Phi(\mu,t) , \quad \Phi(\mu,0) = I.
\]
When $\mu=0$, $\Phi(0,\cdot)$ is the differential of the Hamiltonian flow, so
\[
\lambda (t) = \graf \Phi(0,t) C.
\]
In particular, also using the fact that $N^* W$ is invariant with respect to the involution $C\times C$, we find that
\[
\nu^Q(\gamma) = \dim \ker \mathscr{E} = \dim \bigl( \graf C \Phi(0,1) \bigr) \cap N^* W = \dim \lambda(1)\cap N^*W,
\]
as claimed. The eigenvalue $\lambda$ is negative if and only if $\mu$ belongs to the interval $]-1,0[$, so
\begin{equation}
\label{prifo}
i^Q(\gamma) = \sum_{-1<\mu<0} \dim \bigl( \graf \Phi(\mu,1)C \bigr) \cap N^* W
\end{equation}
(see equation (1.23) in \cite{dui76}). By Proposition 4.1 in \cite{dui76}, the Lagrangian path
\[
[-1,0] \mapsto \mathscr{L}(T^* \R^{2n}), \quad
\mu \mapsto \graf \Phi(\mu,1)C,
\]
has non-trivial intersection with the Lagrangian subspace $N^* W$ for finitely many $\mu\in ]-1,0]$, and the corresponding crossing forms (see (\ref{crform})) are positive definite. Then (\ref{prifo}) and formula (\ref{relmasl}) for the relative Maslov index in the case of regular crossings imply that if $\epsilon>0$ is so small that there are no non-trivial intersections for $\mu \in [-\epsilon,0[$, there holds
\begin{eqnarray}
\label{secfo}
& i^Q(\gamma) = \mu ( \graf \Phi(\cdot,1)C|_{[-1,-\epsilon]}, N^* W), & \\
\label{terfo}
& \mu(\graf \Phi(\cdot,1)C|_{[-\epsilon,0]},N^* W) = \frac{1}{2} \dim \bigl( \graf \Phi(0,1)C \bigr) \cap N^* W = \frac{1}{2} \dim \nu^Q(\gamma). &
\end{eqnarray}
By considering the homotopy
\[
[-1,0]\times [0,1] \rightarrow \mathscr{L}(T^* \R^{2n}) , \quad (\mu,t) \mapsto \graf \Phi(\mu,t) C,
\]
and by using the homotopy and concatenation properties of the relative Maslov index, we obtain from (\ref{secfo}) the identity
\begin{equation}
\label{pezzi}
\begin{split}
i^Q(\gamma) = -\mu (\graf \Phi(-1,\cdot)C|_{[0,1]}, N^* W) + \mu (\graf \Phi(\cdot,0)C|_{[0,1]},N^* W) \\ + \mu (\graf \Phi(0,\cdot)C|_{[0,1]},N^* W) - \mu(\graf \Phi(\cdot,1)C|_{[-\epsilon,0]},N^*W).
\end{split}
\end{equation}
The path $t \mapsto \graf \Phi(-1,t)C$ appearing in the first term can intersect $N^* W$ only for $t=0$, where it coincides with $\graf C = N^* \Delta$, where $\Delta=\Delta_{\R^n}$ is the diagonal in $\R^n \times \R^n$.
By Lemma 4.2 in \cite{dui76}, the corresponding crossing form is non-degenerate and has Morse index equal to
\[
\dim \tau^* ( N^* W \cap N^* \Delta),
\]
where $\tau^*: T^* \R^{2n} \rightarrow \R^{2n}$ is the standard projection. Since $\tau^* ( N^* W \cap N^* \Delta) = W \cap \Delta$, we deduce that this crossing form has signature
\begin{eqnarray*}
\dim N^* W \cap N^* \Delta - 2 \dim W \cap \Delta = \dim W^{\perp} \cap \Delta^{\perp} -  \dim W \cap \Delta
= \dim (W + \Delta)^{\perp} -  \dim W \cap \Delta \\
= 2n - \dim (W + \Delta) -  \dim W \cap \Delta = 2n - \dim W - \dim \Delta = n - \dim W.
\end{eqnarray*}
So by (\ref{relmasl}),
\begin{equation}
\label{duno}
\mu (\graf \Phi(-1,\cdot)C|_{[0,1]}, N^* W) = \frac{1}{2} (n-\dim W).
\end{equation}
Since $\graf \Phi(\mu,0)C = \graf C = N^* \Delta$ does not depend on $\mu$, the second term in (\ref{pezzi}) vanishes,
\begin{equation}
\label{ddue}
\mu (\graf \Phi(\cdot,0)C|_{[0,1]},N^* W) = 0.
\end{equation}
The third term in (\ref{pezzi}) is precisely
\begin{equation}
\label{dtre}
\mu (\graf \Phi(0,\cdot)C|_{[0,1]},N^* W) = \mu(\lambda,N^*W),
\end{equation}
and the last one is computed in (\ref{terfo}). Formulas (\ref{terfo}), (\ref{pezzi}), (\ref{duno}), (\ref{ddue}), and (\ref{dtre}) imply
\[
i^Q(\gamma) = \mu(\lambda,N^* W) + \frac{1}{2} (\dim W-n) - \frac{1}{2} \nu^Q(\gamma),
\]
concluding the proof.
\end{proof} \qed

We conclude this section by reformulating the above result in terms of the Maslov index for solutions of the non-local conormal boundary value Hamiltonian problems introduced in Section \ref{hscbcbs}.

\begin{cor}
\label{indcor}
Assume that the Hamiltonian $H\in C^{\infty}([0,1]\times T^*M)$ satisfies (\ref{convex}-\ref{superlin}), and let $L\in C^{\infty}([0,1]\times TM)$ be its Fenchel dual Lagrangian.
Let $x:[0,1]\rightarrow T^*M$ be a solution of the non-local conormal boundary value Hamiltonian problem (\ref{ham}-\ref{nonlocbdry}), and let $\gamma= \tau^* \circ x : [0,1] \rightarrow M$ be the corresponding solution of (\ref{euler}-\ref{lagnonloc1}-\ref{lagnonloc2}). Then
\[
\nu^Q(x) = \nu^Q(\gamma), \quad \mu^Q(x) = i^Q(\gamma) + \frac{1}{2} \nu^Q(x).
\]
\end{cor}

\begin{rem}
The Tonelli assumptions (\ref{convex}-\ref{superlin}) are needed in order to have a globally defined Lagrangian $L$. Since the Maslov and Morse indexes are local invariants, the above result holds if we just assume the Legendre positivity condition, that is
\[
D_{pp} H(t,q(t),p(t)) > 0 \quad \forall t\in [0,1],
\]
along the Hamiltonian orbit $x(t)=(q(t),p(t))$.
\end{rem}

\section{The Floer complex}

Let us fix a metric $g$ on $M$, with
associated norm $|\cdot|$. We denote by the same symbol the
induced metric on $TM$ and on $T^*M$.
This metric determines an isometry $TM \to T^*M$ and a direct
summand of the vertical tangent bundle $T^v T^*M$, that is, the horizontal
bundle $T^hT^*M$. It also induces a preferred
$\omega$-compatible almost complex structure $J_0$ on $T^*M$, which
in the splitting $TT^*M = T^hT^*M \oplus T^v T^*M$ takes the form
\[
 J_0= \begin{pmatrix} 0 & -I\\
I & 0
\end{pmatrix},
\]
where the horizontal and vertical subbundles are identified by the metric.
In order to have a well-defined Floer complex, we assume that $M$ is
compact, that the submanifold $Q$ of $M\times M$ is also compact, and
that the smooth Hamiltonian $H:[0,1] \times T^*M \rightarrow \R$
satisfies the following conditions:

\begin{enumerate}

\item[(H0)] every solution $x$ of the non-local boundary value
  Hamiltonian problem (\ref{ham}-\ref{nonlocbdry}) is non-degenerate, 
meaning that $\nu^Q(x) = 0$;

\item[(H1)] there exist $h_0>0$ and $h_1 \geq 0$ such that
\[
DH(t,q,p)[\eta] - H(t,q,p) \geq h_0|p|^2- h_1,
\]
for every $(t,q,p)\in [0,1] \times T^*M$ ($\eta$
denotes the Liouville vector field);

\item[(H2)] there exists an $h_2 \geq 0$ such that
\[
|\nabla_q H(t,q,p)|\leq h_2(1+|p|^2), \quad |\nabla_p H(t,q,p)| \leq
h_2(1+ |p|),
\]
for every $(t,q,p) \in [0,1]\times T^*M$ ($\nabla_q$ and $\nabla_p$
denote the horizontal and the vertical components of the gradient).
\end{enumerate}

Condition (H0) holds for a generic choice of $H$, in basically every
reasonable space. Since $M$ is compact, it is easy to show that conditions $(H1)$ and $(H2)$ do not depend on the
choice of the metric on $M$ (it is important here that the exponent of
$|p|$ in the second inequality of (H2) is one unit less than the
corresponding exponent in the first inequality).

We denote by $\mathscr{S}^Q (H)$ the set of solutions of
(\ref{ham}-\ref{nonlocbdry}), which by (H0) is at most countable.
The first variation of the Hamiltonian action functional
\[
\mathbb A_H(x) := \int x^*( \theta - Hdt) = \int_0^1
\bigl(\theta (x'(t)) - H(t,x(t))\bigr) \,dt
\]
on the space of free paths on $T^*M$ is
\begin{equation}\label{eq:fistvariariationham}
d \mathbb A_H(x)[\zeta] = \int_0^1 \omega\bigl(\zeta, x'(t) - X_H(t,
x) \bigr) \,dt + \theta(x(1))[\zeta(1)] -
\theta(x(0))[\zeta(0)],
\end{equation}
where $\zeta $ is a section of $x^*(TT^*M)$. Since the Liouville
one-form $\theta \times \theta$ of $T^*M^2$ vanishes on the conormal
bundle of every submanifold of $M^2$, 
the extremal curves of $\mathbb A_H$ on the space of paths satisfying
(\ref{nonlocbdry}) are precisely the elements of $\mathscr{S}^Q(H)$. 
A first consequence of conditions (H0), (H1), (H2) is that the set of
solutions $x\in \mathscr{S}^Q(H)$ with an upper bound on the action,
$\mathbb{A}_H(x) \leq A$, is finite (see Lemma 1.10 in \cite{as06}).

Let $J$ be a smoothly time-dependent $\omega$-compatible almost complex structure on $T^*M$, meaning that $\omega(J(t,\cdot)\cdot,\cdot)$ is a Riemannian metric on $T^*M$ (notice that the metric almost complex structure $J_0$ is $\omega$-compatible).
Let us consider the Floer equation 
\begin{equation}\label{eq:Floer}
\partial_s u + J(t,u)\bigl( \partial_t u - X_H(t,u) \bigr)=0
\end{equation}
where $u:\R \times [0,1] \to T^*M$, and $(s,t)$ are the coordinates
on the strip $\R \times [0,1]$. It
is a nonlinear first order elliptic PDE, a perturbation of order zero
of the equation for $J$-holomorphic strips on the almost-complex
manifold $(T^*M,J)$. The solutions of
\eqref{eq:Floer} which do not depend on $s$ are the orbits of the
Hamiltonian vector field $X_H$. If $u$ is a solution of
\eqref{eq:Floer}, an integration by parts and formula
\eqref{eq:fistvariariationham} imply the identity
\[
\int_a^b \int_0^1 |\partial_s u|^2\, ds \, dt = \mathbb A_H(u(a,\cdot))- \mathbb
A_H(u(b, \cdot)) + \int_a^b \bigl(\theta(u(s,1))[\partial_s u(s,1)]-
\theta(u(s,0))[\partial_s u(s,0)]\bigr)\, ds.
\]
In particular, if $u$ satisfies also the non-local boundary condition
\begin{equation}\label{eq:nonlocalbdforu}
\left(u(s, 0), -u(s,1)\right) \in N^*Q, \qquad \forall\, s \in \R,
\end{equation}
the fact that the Liouville form vanishes on conormal bundles implies
that
\begin{equation}\label{eq:energyidentityreduced}
\int_a^b \int_0^1 |\partial_s u|^2\, ds \, dt= \mathbb A_H(u(a,\cdot))- \mathbb
A_H(u(b, \cdot)) .
\end{equation}

Given $x^-, x^+ \in \mathscr S^Q(H)$, we denote by $\mathscr M(x^-,
x^+)$ the set of all solutions of
(\ref{eq:Floer}-\ref{eq:nonlocalbdforu}) such that
\[
\lim_{s \to \pm \infty} u(s,t)= x^\pm(t), \quad \forall t\in [0,1].
\]
By elliptic regularity, such solutions are smooth up to the
boundary. Moreover, the condition (H0) implies that the above convergence
of $u(s,t)$ 
to $x^{\pm}(t)$ is exponentially fast in $s$, uniformly with respect
to $t$. Furthermore, (\ref{eq:energyidentityreduced}) implies that the elements $u$ of
$\mathscr{M}(x^-,x^+)$ satisfy the energy identity
\begin{equation}\label{enid}
E(u) := \int_{-\infty}^{+\infty} \int_0^1 |\partial_s u|^2 \, ds \, dt =
\mathbb{A}_H(x^-) - \mathbb{A}_H(x^+).
\end{equation}
In particular, $\mathscr{M}(x^-,x^+)$ is empty whenever
$\mathbb{A}_H(x^-) \leq \mathbb{A}_H(x^+)$ and $x^-\neq x^+$, and it
consists of the only element $u(s,t)=x(t)$ when $x^-=x^+=x$.

A standard transversality argument (see \cite{fhs96}) shows that we can perturb the time-dependent $\omega$-compatible almost complex structure $J$ in order to ensure that the linear operator obtained by
linearizing (\ref{eq:Floer}-\ref{eq:nonlocalbdforu}) along every
solution in $\mathscr M(x^-,x^+)$ is onto, for every
pair $x^-,x^+\in \mathscr{S}^Q(H)$. 
It follows that  $\mathscr M(x^-,
x^+)$ has the structure of a smooth manifold. Theorem 7.42 in \cite{rs95} implies that the dimension of $\mathscr M(x^-,x^+)$ equals
the difference of the Maslov indices of the Hamiltonian orbits $x^-,x^+$:
\[
\dim\, \mathscr M(x^-, x^+)= \mu^Q(x^-) - \mu^Q(x^+).
\]

The manifolds $\mathscr M(x^-, x^+)$ can be oriented in a way which is
{\em coherent with gluing}. This fact is true for more general Lagrangian
intersection problems on symplectic manifolds (see 
\cite{fh93} for periodic orbits and \cite{fooo07} for Lagrangian intersections), but the special situation of conormal boundary
conditions on cotangent bundles allows simpler proofs (see
section 1.4 in \cite{as06}, where the meaning of coherence is also
explained; see also Section 5.2 in \cite{oh97} and Section 5.9 in 
\cite{as08}). 

If the $\omega$-compatible almost complex structure $J$ is $C^0$-close enough to the metric almost complex structure $J_0$, conditions (H1) and (H2) imply that 
the solution spaces $\mathcal{M}(x^-,x^+)$ are pre-compact in
the $C^{\infty}_{\mathrm{loc}}$ topology. In fact, by the energy
identity (\ref{eq:energyidentityreduced}), Lemma 1.12 in \cite{as06}
implies that, setting $u=(q,p)$, $p$ has a uniform bound in
$W^{1,2}([s,s+1] \times [0,1])$. From this fact, Theorem 1.14 in \cite{as06} produces an $L^{\infty}$ bound for the elements of
$\mathcal{M}(x^-,x^+)$ (here is where we need $J$ to be close enough to $J_0$). A $C^1$ bound is then a consequence of the
fact that the bubbling off of $J$-holomorphic spheres and disks cannot occur, the first 
because the symplectic form $\omega$ of $T^*M$ is exact, and the second because the Liouville form -- a
primitive of $\omega$ -- vanishes on conormal bundles. Finally, $C^k$
bounds for all positive integers $k$ follow from elliptic bootstrap.

When $\mu^Q(x^-) - \mu^Q(x^+)=1$, $\mathscr M(x^-,x^+)$ is an oriented
one-dimensional manifold. Since the translation of the $s$ variables
defines a free $\R$-action on it, $\mathscr M(x^-, x^+)$ consists of
lines. Compactness and transversality imply that the number of these
lines is finite. Denoting by $[u]$ the equivalence class of $u$ in the
compact zero-dimensional manifold $\mathscr M(x^-, x^+)/\R$, we define
$\epsilon([u]) \in \{+1,-1\}$ to be $+1$ if the $\R$-action is
orientation preserving on the connected component of $\mathscr
M(x^-,x^+)$ containing $u$, and $-1$ in the opposite case. The integer
$n_F(x^-,x^+)$ is defined as
\[
n_F(x^-,x^+):= \sum_{[u] \in \mathscr M(x^-, x^+)/\R} \epsilon([u]),
\]
If $k$ is an integer, we denote by $F_k^Q(H)$ the free Abelian 
group generated by the
elements $x \in \mathscr S^Q(H)$ with Maslov index $\mu^Q(x)=k$.
These groups need not be finitely generated. The homomorphism
\[
\partial_k : F_k^Q(H) \to F_{k-1}^Q(H)
\]
is defined in terms of the generators by
\[
\partial_k x^- := \sum_{\substack{x^+ \in \mathscr S(H)\\ \mu^Q(x^+)=k-1}} n_F(x^-,x^+)\,x^+, \qquad \forall\, x^-
\in \mathscr S^Q(H), \; \mu^Q(x^-)=k.
\]
The above sum is finite because, as already observed, the set of
elements of $\mathscr{S}^Q(H)$ with an upper action bound is finite. 
A standard gluing argument shows that $\partial_{k-1} \circ \partial_k=0$, so
$\{F_*^Q(H), \partial_* \}$ is a complex of free Abelian groups,
called the {\em Floer complex\/} of $(T^*M,Q,H,J)$. The homology of such a
complex is called the {\em Floer homology\/} of $(T^*M,Q,H,J)$:
\[
HF_k^Q(H,J):= \dfrac{\ker(\partial_k :F_k^Q(H) \to
F^Q_{k-1}(H))}{\ran(\partial_{k+1} :F^Q_{k+1}(H) \to F^Q_k(H))}.
\]
The Floer complex has an $\R$-filtration defined by the action
functional: if $F^{Q,A}_k(H)$ denotes the subgroup of $F_k^Q(H)$
generated by the $x \in \mathscr S^Q(H)$ such that $\mathbb{A}_H(x)<A$,
the boundary operator $\partial_k$ maps $F^{Q,A}_k(H)$ into
$F_{k-1}^{Q,A}(H)$, so $\{F_*^{Q,A}(H), \partial_*\}$ is a subcomplex
(which is finitely generated). 

By changing the orientation data or the almost complex
structure $J$, we obtain an isomorphic chain complex. Therefore, the Floer 
complex of $(T^*M,Q,H)$ is well-defined up to isomorphism.

On the other
hand, a different choice of the Hamiltonian (still satisfying (H0),
(H1), (H2)) produces chain homotopy equivalent complexes. These facts can be
proven by standard homotopy argument in Floer theory, but the
Hamiltonians to be joined have to be chosen close enough, in order to
guarantee compactness (see Theorems 1.12 and 1.13 in \cite{as06}).

\begin{rem}
Conditions (H1) and (H2) do not require $H$ to be convex in $p$, not
even for $|p|$ large. They are used to prove compactness of both the set
of Hamiltonian orbits below a certain action and the set of solutions
of the Floer equation connecting them. They could be replaced by suitable convexity and super-linearity assumptions on $H$. This approach  is taken in the context of generalized Floer homology in \cite{ekp06}. Since this class has a
non-empty intersection with the class of Hamiltonians satisfying (H1)
and (H2), the homotopy type of the Floer complex is the same in both
classes.
\end{rem}

\section{The Morse complex}

In order to define the Morse complex of the Lagrangian action
functional, we shall work with a time-dependent {\em electro-magnetic} Lagrangian, that is  a smooth function $L:[0,1] \times TM \to \R$ of the form
\[
L(t,q,v) = \frac{1}{2} \langle A(t,q) v,v \rangle + \langle \alpha(t,q), v \rangle - V(t,q), \quad \forall t\in [0,1], \; q\in M, \; v\in T_q M,
\]
where $\langle \cdot , \cdot \rangle$ denotes the duality pairing, $A(t,q): T_q M \rightarrow T_q^* M$ is a {\em positive} symmetric linear mapping smoothly depending on $(t,q)$, $\alpha$ is a smoothly time dependent one-form, and $V$ is a smooth function. In particular, $L$ satisfies the classical Tonelli assumptions. As
recalled in Section \ref{morindsec}, these assumptions imply the
equivalence between the Euler-Lagrange equation (\ref{euler}) associated
to $L$ and the Hamiltonian equation (\ref{ham}) associated to its
Fenchel dual $H:[0,1]\times T^*M \rightarrow \R$. Actually, an explicit computation shows that
\[
H(t,q,p) = \frac{1}{2} \langle A(t,q)^{-1} (p - \alpha(t,q)) , p - \alpha(t,q) \rangle + V(t,q), \quad \forall t\in [0,1], \;q\in M, \; p\in T_q^* M,
\]
so $H$ satisfies (H1) and (H2).

By the Legendre transform, the elements $x$ of $\mathscr{S}^Q(H)$ are
in one-to-one correspondence with the solutions $\gamma$ of (\ref{euler})
satisfying the boundary conditions (\ref{lagnonloc1}) and
(\ref{lagnonloc2}). Let $\mathscr{S}^Q(L)$ denote the set of these
$M$-valued curves. The Lagrangian action functional
$\mathbb{S}_L$ is smooth\footnote{In \cite{as06} the first and the third author considered a larger class of Lagrangians, having quadratic growth at infinity. However, under those assumptions the Lagrangian action functional $\mathbb{S}_L$ is only $C^1$ and twice Gateaux-differentiable on the Hilbert manifold $W^{1,2}([0,1],M)$, a fact that was overlooked in \cite{as06}. A Morse theory under these weaker regularity conditions is possible (see for instance \cite{ben91}), but we prefer to avoid these technical difficulties and work with an electro-magnetic Lagrangian $L$, for which $\mathbb{S}_L$ is indeed smooth.} on the Hilbert
manifold $W^{1,2}_Q([0,1],M)$ consisting of the absolutely continuous curves
$\gamma:[0,1]\rightarrow M$ whose derivative is square integrable and
such that $(\gamma(0),\gamma(1)) \in Q$. 
The elements of $\mathscr{S}^Q(L)$ are
precisely the critical points of the restriction of $\mathbb{S}_L$ to such a manifold,
and condition (H0) is equivalent to:
\begin{enumerate}
\item[(L0)] all the critical points $\gamma\in \mathscr{S}^Q(L)$ of the restriction of $\mathbb{S}_L$ to
  $W^{1,2}_Q([0,1],M)$ are non-degenerate,
\end{enumerate}
meaning that the bounded self-adjoint operator on $T_{\gamma}
W^{1,2}_Q([0,1],M)$ representing the second differential of
$\mathbb{S}_L$ at $\gamma$ with respect to a $W^{1,2}$ inner product
is an isomorphism. Under this assumption, Corollary \ref{indcor} implies that
the Morse index $i^Q(\gamma)$ of $\gamma\in \mathscr{S}^Q(L)$, seen as a
critical point of the restriction of $\mathbb{S}_L$ to $W^{1,2}_Q([0,1],M)$, coincides
with the Maslov index $\mu^Q(x)$ of the corresponding element $x\in
\mathscr{S}^Q(H)$. 

The Lagrangian $L$ is bounded from below and so is the action functional
$\mathbb S_L$. The metric of the compact manifold $M$ induces a
complete Riemannian structure on the Hilbert manifold $W^{1,2}_Q([0,1],
M)$, namely
\[
\langle\langle \xi, \zeta \rangle\rangle := \int_0^1 \bigl( g
(\nabla_t \xi, \nabla_t \zeta) + g( \xi,
\zeta ) \bigr)\, dt, \qquad \forall\, \gamma \in
W^{1,2}_Q([0,1], M), \; \forall \xi, \zeta \in T_\gamma W^{1,2}_Q([0,1], M),
\]
where $\nabla_t$ denotes the Levi-Civita covariant derivative along
$\gamma$. The functional $\mathbb{S}_L$ satisfies
the {\em Palais-Smale condition} on the Riemannian manifold
$W^{1,2}_Q([0,1], M)$, that is every sequence $(\gamma_h)\subset 
W^{1,2}_Q([0,1], M)$ such that $\mathbb{S}_L(\gamma_h)$ is bounded and
$\| \nabla \mathbb{S}_L(\gamma_h)\|$ is infinitesimal has a 
subsequence which converges in the $W^{1,2}$ topology (see e.g.\
Appendix A in \cite{af07}).  

Therefore, the functional $\mathbb{S}_L$ is smooth, bounded from below, 
has non-degenerate critical points
with finite Morse index, and satisfies the Palais-Smale
condition on the complete Riemannian manifold $W^{1,2}_Q([0,1],
M)$. Under these assumptions, the Morse complex of $\mathbb{S}_L$ on
$W^{1,2}_Q([0,1],M)$ is well-defined (up to chain isomorphism) and its
homology is isomorphic to the singular homology of  $W^{1,2}_Q([0,1],
M)$. The details of the construction are contained in
\cite{ama06m}. Here we just state the results and fix the notation.

Let $M_k^Q(\mathbb{S}_L)$ be the free Abelian group generated by the
elements $\gamma$ of $\mathscr{S}^Q(L)$ of Morse index
$i^Q(\gamma)=k$. Up to a perturbation of the Riemannian metric on $W^{1,2}_Q([0,1],
M)$, the unstable and stable manifolds $W^u(\gamma^-)$ and
$W^s(\gamma^+)$ of  
$\gamma^-$ and $\gamma^+$ with respect to the negative gradient flow of
$\mathbb{S}_L$ on $W^{1,2}_Q([0,1],M)$ have transverse intersections
of dimension $i^Q(\gamma^-) - i^Q(\gamma^+)$, for every pair of
critical points $\gamma^-,\gamma^+$. An arbitrary choice of 
orientation for the (finite-dimensional) unstable manifold of each critical point induces
an orientation of all these intersections. When
$i^Q(\gamma^-) - i^Q(\gamma^+)=1$, such an intersection consists of
finitely many oriented lines. The integer $n_M(\gamma^-,\gamma^+)$ is
defined to be the number of those lines where the orientation agrees
with the direction of the negative gradient flow minus the number of
the other lines. Such integers are the coefficients of the
homomorphisms
\[
\partial_k : M_k^Q(\mathbb{S}_L) \rightarrow M_{k-1}^Q(\mathbb{S}_L),
\quad \partial_k \gamma^- = \sum_{\substack{\gamma^+ \in
    \mathscr{S}^Q(L) \\ i^Q(\gamma^+) = k-1}} n_M(\gamma^-,\gamma^+)\,
\gamma^+,
\]
defined in terms of the generators $\gamma^- \in \mathscr{S}^Q(L)$,
$i^Q(\gamma^-) = k$. This sequence of homomorphisms can be identified with the
boundary operator associated to a cellular filtration of  $W^{1,2}_Q([0,1],
M)$ induced by the negative gradient flow of
$\mathbb{S}_L$. Therefore, $\{M_*^Q(\mathbb{S}_L), \partial_*\}$ is a
chain complex of free Abelian groups, called the {\em Morse complex} of
$\mathbb{S}_L$ on $W^{1,2}_Q([0,1],M)$, and its homology is isomorphic
to the singular homology of $W^{1,2}_Q([0,1],M)$. Changing the
(complete) Riemannian metric on $W^{1,2}_Q([0,1],M)$ produces a chain
isomorphic Morse complex. The Morse complex is filtered by the action
level, and the homology of the subcomplex generated by all elements
$\gamma\in \mathscr{S}^Q(L)$ with $\mathbb{S}_L(\gamma) < A$ is
isomorphic to the singular homology of the sublevel
\[
\set{\gamma \in W^{1,2}_Q([0,1],M)}{\mathbb{S}_L(\gamma) < A}.
\]
The embedding
of $W^{1,2}_Q([0,1],M)$ into the space $P_Q([0,1],M)$ of {\em continuous}
curves $\gamma:[0,1]\rightarrow M$ such that $(\gamma(0),\gamma(1))\in
Q$ is a homotopy equivalence. Therefore, the homology of the above
Morse complex is isomorphic to the singular homology of the path space
$P_Q([0,1],M)$.

\section{The isomorphism between the Morse and the Floer complex}

We are now ready to state and prove the main result of this paper. Here $M$ is a compact manifold and $Q$ is a compact submanifold of $M\times M$.

\begin{thm}\label{thm:isomorphism}
Let $L \in C^\infty([0,1] \times TM)$ be a time-dependent electro-magnetic 
Lagrangian satisfying condition (L0). Let $H\in C^{\infty} ([0,1] \times T^* M)$ be its Fenchel-dual Hamiltonian. 
Then there is a chain complex isomorphism
\[
\Theta : \{ M_*^Q(\mathbb S_L), \partial_*\} \longrightarrow \{F_*^Q(H), \partial_*\}
\]
uniquely determined up to chain homotopy, having the form
\[
\Theta \gamma = \sum_{\substack{x \in \mathscr S^Q(H)\\
\mu^Q(x)=i^Q(\gamma)}} n_{\Theta}(\gamma,x)\,x, \qquad \forall \gamma \in \mathscr S^Q (L),
\]
where $n_{\Theta}(\gamma,x)=0$ if $\mathbb S_L(\gamma) \leq \mathbb A_H (x)$ unless
$\gamma$ and $x$ correspond to the same solution, in which case
$n_{\Theta}(\gamma,x)=\pm 1$. In particular, $\Theta$ respects the action filtrations of the Morse and the Floer complexes.
\end{thm}

\begin{proof}
Let $\gamma\in \mathscr{S}^Q (L)$ and $x\in \mathscr{S}^Q(H)$. Let $\mathscr{M}(\gamma,x)$ be the space of all $T^*M$-valued maps on the half-strip $[0,+\infty[ \times [0,1] $ solving the Floer equation (\ref{eq:Floer}) with the asymptotic condition
\begin{equation}
\label{asymp}
\lim_{s\rightarrow +\infty} u(s,t) = x(t),
\end{equation}
and the boundary conditions
\begin{eqnarray}
\label{ccc1}
(u(s,0),-u(s,1)) \in N^* Q, \qquad \forall s\geq 0, \\
\label{ccc2}
\tau^* \circ u(0,\cdot) \in W^u(\gamma),
\end{eqnarray} 
where $\tau^*:T^*M \rightarrow M$ is the standard projection and $W^u(\gamma)$ denotes the unstable manifold of $\gamma$ with respect to the negative gradient flow of $\mathbb{S}_L$ on $W^{1,2}_Q([0,1],M)$. By elliptic regularity, these maps are smooth on $]0,+\infty[ \times [0,1]$ and continuous on $[0,1] \times [0,+\infty[$. Actually, the fact that $\tau^* \circ u(0,\cdot)$ is in $W^{1,2}([0,1])$ implies that $u\in W^{3/2,2}(]0,S[\times ]0,1[)$ for every $S>0$, and, in particular, $u$ is H\"older continuous up to the boundary. 

The proof of the energy estimate for the elements of $\mathscr{M}(\gamma,x)$ is based on the following immediate consequence of the Fenchel formula (\ref{fenchel}) and of (\ref{equality}):

\begin{lem}
\label{confronto}
If $x=(q,p):[0,1] \to T^*M$ is continuous, with $q$ of class
$W^{1,2}$, then
\[
\mathbb A_H (x) \leq  \mathbb S_L (q),
\]
the equality holding if and only if the curves $(q,q')$ and $(q,p)$ are related by the Legendre transform, that is, $(t,q(t), q'(t))= \mathcal L(t, q(t), q'(t))$ for every $t \in [0,1]$. In particular, the
Hamiltonian and the Lagrangian action coincide on corresponding solutions of the two systems.
\end{lem}

In fact, if $u\in \mathscr{M}(\gamma,x)$, the above Lemma, together with
(\ref{eq:energyidentityreduced}) (which holds because of (\ref{ccc1})), (\ref{asymp}), and (\ref{ccc2}) imply that
\begin{equation}
\label{enest}
\begin{split}
E(u) := \int_0^{+\infty} \int_0^1 |\partial_s u|^2\, ds\, dt = \mathbb{A}_H(u(0,\cdot)) - \mathbb{A}_H(x) \\ \leq \mathbb{S}_L( \tau^* \circ  u(0,\cdot)) - \mathbb{A}_H(x) \leq \mathbb{S}_L(\gamma) -  \mathbb{A}_H(x).
\end{split} \end{equation}
This energy estimate allows to prove that $\mathscr{M}(\gamma,x)$ is pre-compact in $C^{\infty}_{\mathrm{loc}}$, as in Section 1.5 of \cite{as06}. It also implies that:
\begin{enumerate}
\item[(E1)] $\mathscr{M}(\gamma,x)$ is empty if either $\mathbb{S}_L(\gamma) < \mathbb{A}_H(x)$ or $\mathbb{S}_L(\gamma) = \mathbb{A}_H(x)$ but $\gamma$ and $x$ do not correspond to the same solution;
\item[(E2)] $\mathscr{M}(\gamma,x)$ only consists of the element $u(s,t) = x(t)$ if $\gamma$ and $x$ correspond to the same solution. 
\end{enumerate}

The computation of the dimension of $\mathscr{M}(\gamma,x)$ is based on the following linear result, which is a particular case of Theorem 5.24 in \cite{as08}:

\begin{prop}
Let $A:[0,+\infty] \times [0,1] \rightarrow \mathrm{L_s} (\R^{2n})$ be a continuous map into the space of symmetric linear endomorphisms of $\R^{2n}$. Let $V$ and $W$ be linear subspaces of $\R^n$ and $\R^n \times \R^n$, respectively. We assume that $W$ and $V\times V$ are partially orthogonal, meaning that their quotients by the common intersection $W \cap (V\times V)$ are orthogonal in the quotient space. 
We assume that the path $G$ of symplectic automorphisms of $\R^{2n}$ 
defined by
\[
G'(t) = J_0 A(+\infty,t) G(t), \quad G(0) = I, \quad \mbox{where} \; J_0 = \begin{pmatrix} 0 & -I\\
I & 0 \end{pmatrix},
\]
satisfies
\[
\graf G(1)C \cap N^* W = (0).
\]
Then, for every $p\in ]1,+\infty[$, the bounded linear operator
\[
 v \mapsto \partial_s v + J_0 \partial_t v + A(s,t) v
\]
from the Banach space
\[
\set{v\in W^{1,p}(]0,+\infty[ \times ]0,1[, \R^{2n})}{ v(0,t) \in N^* V \; \forall t\in [0,1], \; (v(s,0),-v(s,1)) \in N^* W \; \forall s\geq 0} 
\] 
to the Banach space $L^p( ]0,+\infty[ \times ]0,1[), \R^{2n})$ is Fredholm of index
\begin{equation}
\label{indice}
\frac{n}{2} - \mu( \graf G(\cdot) C,N^*W) - \frac{1}{2} ( \dim W + 2 \dim V - 2 \dim W \cap (V\times V)).
\end{equation}
\end{prop}

If we linearize the problem given by
(\ref{eq:Floer}-\ref{asymp}-\ref{ccc1}), with the condition    
(\ref{ccc2}) replaced by the condition that $\tau^* \circ u(0,\cdot)$ should be a given curve on $M$, we obtain an operator of the kind introduced in the above Proposition, where $V=(0)$, $\dim W = \dim Q$, and $G$ is the linearization of the Hamiltonian flow along $x$. By Proposition \ref{locfor},
\[
\mu(\graf G(\cdot)C,N^*W) = \mu^Q(x) - \frac{1}{2} (\dim Q - n),
\]
so by (\ref{indice}) this operator has index $-\mu^Q(x)$. Since (\ref{ccc2}) requires that the curve $\tau^* \circ u(0,\cdot)$ varies within a manifold of dimension $i^Q(\gamma)$, the linearization of the full problem  (\ref{eq:Floer}-\ref{asymp}-\ref{ccc1}-\ref{ccc2}) produces an operator of index $i^Q(\gamma) - \mu^Q(x)$. By perturbing the time-dependent almost complex structure $J$ 
and the metric on $W^{1,2}_Q([0,1],M)$, we may assume that this operator is onto for every $u\in \mathscr{M}(\gamma,x)$ and every $\gamma\in \mathscr{S}^Q(L)$, $x\in \mathscr{S}^Q(H)$, except for the case in which
$\gamma$ and $x$ correspond to the same solution. In the latter case, by (E2),
$\mathscr{M}(\gamma,x)$ consists of the only map $u(s,t)=x(t)$,
and the corresponding linear operator is not affected by the above perturbations. However, in this case this operator is automatically onto. The proof of this fact is based on the following consequence of Lemma \ref{confronto}, and is analogous to the proof of Proposition 3.7 in \cite{as06}.
 
 \begin{lem}
 If $x\in \mathscr{S}^Q(H)$ and $\gamma = \tau^* \circ x$, then
 \[
 d^2 \mathbb{A}_H(x) [\xi,\xi]  \leq d^2 \mathbb{S}_L (\gamma) [D\tau^* (x) [\xi],  D\tau^* (x) [\xi]],
 \]
 for every section $\xi$ of $x^*(TT^*M)$.
 \end{lem}
 
We conclude that, whenever $\mathscr{M}(\gamma,x)$ is non-empty, it is a manifold of dimension
\[
\dim \mathscr{M}(\gamma,x) = i^Q(\gamma) - \mu^Q(x).
\]
See Section 3.1 in \cite{as06} for more details on the arguments just sketched

When $i^Q(\gamma) = \mu^Q(x)$, compactness and transversality imply that $\mathscr{M}(\gamma,x)$ is a finite set. Each of its points carries an orientation-sign $\pm 1$, as explained in Section 3.2 of \cite{as06}. The sum of these contributions defines the integer $n_{\Theta}(\gamma,x)$. A standard gluing argument shows that the homomorphism 
\[
\Theta : \{ M_*^Q(\mathbb S_L), \partial_*\} \longrightarrow \{F_*^Q(H), \partial_*\}, \quad \Theta \gamma = \sum_{\substack{x\in \mathscr{S}^Q(H)\\ \mu^Q(x) = i^Q(\gamma)}} n_{\Theta}(\gamma,x)\, x, \quad \forall \gamma\in \mathscr{S}^Q(L),
\]
is a chain map. By (E1) such a chain map preserves the action filtration. In other words, if we order the elements of $\mathscr{S}^Q(L)$ and $\mathscr{S}^Q(H)$ -- that is, the generators of the Morse and the Floer complexes -- by increasing action, the homomorphism $\Theta$ is upper-triangular with respect to these ordered sets of generators. Moreover, by (E2) the diagonal elements of $\Theta$ are $\pm 1$. These facts imply that $\Theta$ is an isomorphism, which concludes the proof. 
\end{proof} \qed

\begin{cor}
Let $H:[0,1]\times T^*M \rightarrow \R$ be a Hamiltonian satisfying (H0), (H1), and (H2). Then the  homology of the Floer complex of $(T^*M,Q,H)$ is isomorphic to the singular homology of the path space $P_Q([0,1],M)$.
\end{cor}

\end{document}